\documentclass[12pt]{amsart}

\theoremstyle{thmit} 
\newtheorem{thm}{Theorem}[section]
\newtheorem{lem}[thm]{Lemma}
\newtheorem{conjecture}[thm]{Conjecture}
\newtheorem{cor}[thm]{Corollary}
\newtheorem{prop}[thm]{Proposition}
\newtheorem{definition}[thm]{Definition}

\theoremstyle{thmrm} 

\newtheorem{rem}{Remark}

\usepackage{mathrsfs}
\usepackage{tikz-cd}
\usepackage{graphicx}
\usepackage[colorlinks=true,linkcolor=blue,citecolor=blue]{hyperref}
\usepackage{amsmath,amsthm, amssymb, mathtools}
\usepackage{stmaryrd}
\usepackage{eucal}
\usepackage[T1]{fontenc}
\DeclareSymbolFont{bbold}{U}{bbold}{m}{n}
\DeclareSymbolFontAlphabet{\mathbbold}{bbold}
\usepackage{enumerate}

\numberwithin{equation}{section}

\newcommand\EatDot[1]{}

 \newcommand{\cF}{\mathcal{F}}
\newcommand{\cG}{\mathcal{G}} 
 
\newcommand{\cK}{\mathcal{K}} \newcommand{\cL}{\mathcal{L}}
 \newcommand{\cN}{\mathcal{N}}
\newcommand{\cO}{\mathcal{O}} 
 \newcommand{\cR}{\mathcal{R}}

\newcommand{\cW}{\mathcal{W}} 
 
 \newcommand{\fR}{\mathfrak{R}}

 \newcommand{\bF}{\mathbb{F}}

\newcommand{\bK}{\mathbb{K}}

\newcommand{\bQ}{\mathbb{Q}}

 \newcommand{\bZ}{\mathbb{Z}}

\newcommand{\fl}{\mathfrak{l}} \newcommand{\fp}{\mathfrak{p}}

\newcommand{\fm}{\mathfrak{m}}

\newcommand{\fq}{\mathfrak{q}}

\newcommand{\cyc}{{\mathrm{cyc}}}

\newcommand{\Gal}{{\mathrm{Gal}}}

\newcommand{\Coker}{{\mathrm{Coker}}}
\newcommand{\Sel}{{\mathrm{Sel}}}

\newcommand{\Hom}{\mathrm{Hom}}

\newcommand{\ot}{\otimes}

\newcommand{\Image}{\mathrm{Im}}

\newcommand{\pr}{^{\prime}}

\newcommand{\ds}{\displaystyle}

\newcommand{\op}{\oplus}

\usepackage{mdframed}

\newcommand{\T}{\text}
\newcommand{\Sss}{S^\mathrm{ss}}
\newcommand{\Sbad}{S^\mathrm{bad}}
\newcommand{\bfp}{{\bar{\fp}}}
\newcommand{\bfq}{{\bar{\fq}}}
\newcommand{\ur}{{\mathrm{ur}}}
\newcommand{\Tor}{{\mathrm{Tor}}}
\newcommand{\Tr}{{\mathrm{Tr}}}
\newcommand{\Epinf}{E_{p^{\infty}}}
\newcommand{\wEpinf}{\widehat{E}_{p^{\infty}}}
\newcommand{\Pd}{^{\wedge}}
\newcommand{\hooklongrightarrow}{\lhook\joinrel\longrightarrow}

\newtheorem{corollary 4.9}{Corollary 4.9}

\newcommand{\depth}{\mathrm{depth}}

\usepackage{comment}
\usepackage{graphicx}
\newcommand{\cf}{{\textit{cf. }}}

\title[Residual supersingular Iwasawa theory]{Residual supersingular Iwasawa theory over quadratic imaginary fields}
\author[P. Hamidi]{Parham Hamidi}
\address{Department of Mathematics, The University of British Columbia,\\
	Room 121, 1984 Mathematics Road\\
	Vancouver, BC\\
	Canada V6T 1Z2}
\email{phamidi@math.ubc.ca}
\begin{document}
	
	\maketitle
	\begin{abstract}
		Let $ p $ be an odd prime. 
		Let $ E $ be an elliptic curve defined over a quadratic imaginary field where $ p $ splits completely. Suppose $ E $ has supersingular reduction at primes above $ p $.
		Under appropriate hypotheses, we extend the results of \cite{SujFil} to $ \bZ_p^{2} $-extensions. 
		{We  define and study the fine double-signed residual Selmer groups in these settings.}
		We  prove that for two residually isomorphic elliptic curves, the vanishing of the signed $ \mu $-invariants of one elliptic curve implies the  vanishing of the signed $ \mu $-invariants of the other.
		Finally, we  show that the Pontryagin dual of the Selmer group and the double-signed Selmer groups have no non-trivial pseudo-null submodules for these extensions.
	\end{abstract}
	\section{Introduction}\label{sec:Introducion}
	{
		Consider two elliptic curves defined over $ \bQ $ with good ordinary reduction at prime $ p $ whose residual Galois representations are isomorphic.
		Greenberg and Vatsal in \cite{GreenbergVatsal} showed that the vanishing of the $ \mu $-invariant attached to the Pontryagin dual of the Selmer group over the cyclotomic $ \bZ_p $-extension  of one curve implies the vanishing of the $ \mu $-invariant for the other.
		Their work uses an auxiliary Selmer group called\textit{ nonprimitive dual Selmer group} which has the same $ \mu $-invariant as the Pontryagin dual of the Selmer group.
		{ 
			For an elliptic curve $ E/\bQ $ with good supersingular reduction at $ p $, the Selmer group over the cyclotomic extension is no longer cotorsion.
			Kobayashi in \cite{Kobayashi} defined the signed Selmer group over the cyclotomic $ \bZ_p $-extension using moderately stronger local conditions at $ p $ and showed that they are cotorsion over Iwasawa modules.
		}
		A similar approach has been taken in the study of elliptic curves with supersingular reduction for the signed Selmer groups. 
		The $ \mu $-invariants of the Pontryagin dual of the signed Selmer groups are referred to as the \textit{signed $ \mu $-invariants}.
		An analogue of Greenberg--Vatsal {was proved} by B. D. Kim in \cite{Kim2009} for the signed $ \mu $-invariants over cyclotomic $ \bZ_p $-extension using nonprimitive dual Selmer groups.
		Later, in \cite{SujFil} the authors used a new technique to work with residual representations of elliptic curves and they improved upon the results of Kim in \cite{Kim2009}.
		{They defined a new Selmer group called \textit{the fine residual signed Selmer group} 
			and studied its structure as an Iwasawa module, in particular, the vanishing of the signed $ \mu $-invariants (\textit{cf.} \cite[Theorem 4.12]{SujFil}).}

		In this paper, we follow the strategy of \cite{SujFil} to prove new results for the signed $ \mu $-invariants of $ \bZ_p^{2} $-extensions of quadratic imaginary fields. 
		However, the Iwasawa theory of $ \bZ_p^{2} $-extensions has additional technicalities compared to the cyclotomic  $ \bZ_p$-extensions.
		We  define the fine signed residual Selmer groups in these settings (\cf Definition \ref{def:signed residual Selmer}). 
		In Proposition \ref{prop:fine selmer only depends on residual reps} we show that these residual singed Selmer groups depend only on the isomorphism class of the residual Galois representation of elliptic curves.
		Therefore, these groups provide a natural method to study congruences of elliptic curves.
		We then relate the structure of the fine signed residual  Selmer groups to that of the signed Selmer groups as Iwasawa modules (\cf Proposition \ref{prop:cR andmathscrR have the same corank}).
		Theorem \ref{thm:TFAE} relates the module structure of the fine signed residual Selmer groups to the vanishing of the signed $ \mu $-invariants.
		We use this theorem to show that if two elliptic curves have isomorphic residual representations, then vanishing of the signed $\mu $-invariants for one implies the same for the other.
		Furthermore, we relate these results to analogous results over cyclotomic extensions proved in \cite{SujFil} (\cf Corollary \ref{cor:TFAE-cyc implied by TFAE}).
		
		Moreover, we prove in  that the Pontryagin dual of the signed Selmer groups and the classical Selmer group have no non-zero pseudo-null submodules (\cf Theorem \ref{thm:no no pseudo-null}).
		This is an important property for an Iwasawa module, since the structure of finitely generated modules over commutative Iwasawa algebras is known up to pseudo-isomorphism.
		Hence, whenever such modules have no non-zero pseudo-null submodules we have a better understanding of their  structure.
		To this this, we invoke Auslander--Buchsbaum--Serre formula and compute the depth of the signed Selmer groups by studying Galois cohomology of these groups.
		
		This paper consists of five sections including this introductory section and it is organized as follows.
		{ 
			In section \ref{sec:preliminaries}, we introduce preliminary definitions, notations, and assumptions.
		}
		Our goal in section \ref{sec:local and global cohomology calculations} is to prove some of the essential components of the  proof of our main results.
		{In section \ref{sec:Signed Selmer and fine residual Selmer groups as Iwasawa modules}, we prove Theorem \ref{thm:TFAE} and record some important consequences of it.}
		Finally, in section \ref{sec:Pseudo-null submodules}  we use purely algebraic tools to compute the depth of signed Selmer groups and prove Theorem \ref{thm:no no pseudo-null} which states that the Pontryagin dual of the Selmer group and the double-signed Selmer groups have no non-trivial pseudo-null submodules.
		
		\section{Preliminaries}\label{sec:preliminaries}
		Let $ E/L $ be an elliptic curve defined over a quadratic imaginary number field  $L/\bQ$. 
		Suppose $ p $ is an odd prime and let $ S_p $ denote the set of all primes of $ L $ over $ p $.
		Denote the finite set of primes above $ p $ where $ E/L $ has supersingular reduction by $ \Sss \subseteq S_p $ and
		let $ \Sbad $ denote the finite set of primes in $ L $ where $ E $ has bad reduction. 
		Let $ S $ be the disjoint union of the sets $ S_p $ and $ \Sbad $.
		For any field extension $ \cL/L $, let $ S_p(\cL) $ be the set primes in $ \cL $ above the set $ S_p $. 
		Similarly, let $ S^{*}(\cL) $ for $ * \in \{ \emptyset, \ \T{bad}, \ \T{ss} \} $ denote the set of primes in $ \cL $ above the corresponding finite set $ S^{*} $.
		Throughout this article we assume the following hypotheses, which we refer to as \hyperref[Hyp1]{Hyp 1}.
		\begin{enumerate}[{\textbf{Hyp 1}} (i):]\label{Hyp1}
			\item The prime $ p $ is odd. 
			\item The number field $L/\bQ$ is a quadratic imaginary field extension. 
			Furthermore, assume that $ p $ splits completely in $ L/\bQ $. We denote the primes of $L$ over $ p $ by $\fp$ and $ \bar{\fp}$. 
			In particular, $ \fp \neq \bar{\fp} $ and $ L_\fp = L_\bfp = \bQ_p $.
			\item\label{allss} The elliptic curve $E/L$ has good supersingular reduction at both primes above prime $ p $. 
			Therefore, the sets $  \Sss = S_p = \{ \fp, \bar{\fp} \} $.  
			\item We have $ a_\fp=1+ N_{L/\bQ}(\fp) - \#\widetilde{E}_\fp(\cO_L/\fp)= a_\bfp = 0 $ where
			$ N_{L/\bQ} : L \to \bQ  $ is the usual norm map and $\widetilde{E}_\fp$ is the reduced curve modulo $ \fp $.
		\end{enumerate} 
		
		Let $ L_\cyc/ L $ denote the cyclotomic  $ \bZ_p $-extension of $ L $ and let $ L_\infty/ L $ denote the compositum of all $ \bZ_p $-extensions of $ L $. 
		Leopoldt's conjecture implies that $ \Gal(L_\infty/L) \cong \bZ_p^2 $.
		Note that Leopoldt's conjecture is known for abelian extensions of rational numbers, and therefore it is not a hypothesis in this case (\textit{cf.} Corollary 5.32 and Theorem 13.4 of \cite{Wash1}).
		{Let $ L^{S} $ be the maximal algebraic extension of $ L $ unramified outside of the primes inside $ S $ and $ G_{L}^{S} $ denote the Galois group of the extension $ L^{S}/L $. Note that $L_\infty \subset L^{S}$.
			Let 
			\begin{align*}
				G_{\infty}^{S}:=\Gal(L^{S}/L_\infty),  \ G := \Gal(L_\infty/ L) \cong \bZ_p^{2}, \ \T{and }   \Gamma :=  \Gal(L_\cyc/ L) \cong \bZ_p.
			\end{align*}
		}
			Given a prime $ w $ in $ S_p(L_\infty) $, by abuse of notation, let the prime below $  w$ in $ L_n $  be again denoted by $ w $.
			It should be clear from the context whether the prime $ w $ is in $ L_n $ for some $ n\geq 0 $ or is in $L_\infty  $.
			Finally, for an abelian group  $ M $ and an integer $ t \geq 1 $ we write $ M_t $ for the subgroup of elements of $ M $ that are annihilated by $ t $.
			Moreover, for a prime $ p $, we define the $p$-primary torsion subgroup of $M$, which we denote by $ M_{p^{\infty}} $, as 
			\[
			M_{p^{\infty}}:= \ds\cup_{i\geq 1} M_{p^{i}}.
			\] 
			
			\subsection{Plus and minus norm groups}
			Suppose $ K/\bQ_p $ is a finite unramified extension of $\bQ_p$. 
			Denote the cyclotomic (\textit{resp}. the unramified) $ \bZ_p $-extension of $ K $ by $ K_\cyc $ (\textit{resp}. $K^{\ur} $)
			{which is totally ramified.}
				For any integer $ r\geq 0 $, let $ K_r $ be the unique intermediate field extension of $ K_\cyc/K $ such that $ \Gal(K_r/K)= \bZ/p^{r}\bZ $ and for any integer $ l\geq 0 $, let $ K^{(l)} $ be the unique intermediate field extension of $ K^{\ur}/K $ such that $ \Gal(K^{l}/K)$ is isomorphic to $\bZ/p^{l}\bZ $.
				Further, let $ K^{(l)}_r$ be $ K^{l} K_r $, so $ \Gal(K^{(l)}_r/K)$ is $\bZ/p^{l}\bZ \times \bZ/p^{r}\bZ  $ and  
				{
					\begin{align*}
						K^{\ur}_{\cyc} = K^{\ur}K_{\cyc} = \cup_{r,l\geq 0}K^{l} K_r,
					\end{align*}
					with $ \Gal(K^{\ur}_{\cyc}/K)$ equal to $\bZ_p^{2} $.
				}
				Let $ \mu_{p^{r}} $ denote the $ p^{r} $-th {roots} of unity and let $ \bK^{(l)}_r $ to be $ K^{(l)}(\mu_{p^{r+1}})  $ for $ l \geq 0 $ and $ r\geq -1 $.
				For all $ l,r \geq 0 $, denote
				\begin{align}\label{not:delta}
					\Delta := \Gal(K(\mu_p)/K)= \Gal(\bK^{(0)}_{0}/K^{(0)}_{0})\cong \Gal(\bK^{l}_{r}/K^{l}_{r}).
				\end{align}
				In this article, with abuse of notation, we write 
				$ (\bK^{(l)}_r)^{\Delta} $ is equal to $K^{(l)}_{r} $.

				In what follows, let $ \widehat{E}(\fm_\cL) $ by $ \widehat{E}(\cL) $ where $ \cL/\bQ_p $ is any local field over $ \bQ_p $ and $ \fm_\cL $ is the maximal ideal of the valuation ring  of $ \cL $. 
				Denote the ring of integers of $ \cL  $ by $\cO_\cL   $.
			Suppose $ E/\bQ_p $ is an elliptic curve  and let $\widehat{E}  $ denote the formal group over $ \bZ_p $
			associated to the minimal model of $ E $ over $ \bQ_p $.
			\begin{definition}\label{def:plusminusnormgroups}
				Following \cite[Definition 8.16]{Kobayashi},
				define the \textbf{plus and minus norm groups} as 
				\begin{small}
					\begin{align*}
						\widehat{E}^{\pm}(\bK^{(l)}_r) & := \left\{ P \in \widehat{E}(\bK^{(l)}_r) \mid  \Tr^{r}_{m+1} (P) \in  \widehat{E}(\bK^{(l)}_m), -1 \leq m \leq r-1 \T{ s.t. } (-1)^{m}= \pm 1 \right\},
						\\
						E^{\pm}(\bK^{(l)}_r) & := \left\{ P \in E(\bK^{(l)}_r) \mid  \Tr^{r}_{m+1} (P) \in  E(\bK^{(l)}_m),  -1 \leq m \leq r-1 \T{ s.t. } (-1)^{m}= \pm 1\right\},\\
						\widehat{E}^{\pm}(K^{(l)}_r) &:= \left( \widehat{E}^{\pm}(\bK^{(l)}_r) \right)^{\Delta}, \T{ and } E^{\pm}(K^{(l)}_r) := \left( E^{\pm}(\bK^{(l)}_r) \right)^{\Delta},
					\end{align*}
				\end{small}
				\[
				\widehat{E}^{\pm}(\bK^{(\ur)}_\cyc) := \bigcup_{r,l\geq 0} \widehat{E}^{\pm}(\bK^{(l)}_r), \T{ and } E^{\pm}(\bK^{\ur}_\cyc) := \bigcup_{r,l\geq 0} \widehat{E}^{\pm}(\bK^{(l)}_r).
				\]
			\end{definition}
				Suppose $ E/\bQ_p $ is an elliptic curve with good reduction at $ p $ and such that $ a_p=0 $. 
				{Then, by Proposition 8.7 of \cite{Kobayashi}, for all $ l,r \geq 0 $, $ \widehat{E}(\bK^{(l)}_{r}) $ has no non-trivial $ p $-torsion points.}
				This implies that 
				\begin{align}\label{eq:Ep=0}
					E(\bK^{(l)}_{r})_{p}=E(K^{(l)}_{r})_{p} = \{ 0 \}.
				\end{align}
			\subsection{Signed Kummer maps}\label{sec:kummer-maps}
				Suppose $ L $ and $ S $ are as \hyperref[Hyp1]{Hyp 1}. 
				For any $ n \geq 0 $, let $ L_n $ be the unique sub-extension of $ L_\infty/L $ such that 
				\begin{align}\label{def:Ln}
					\Gal(L_n/L) \cong (\bZ/p^{n}\bZ)^{2}, \quad L= L_0 , \quad L_\infty := \ds\cup_{n\geq 0} L_n.
				\end{align}		
				For simplicity, let $ G_{n}^{S}$ denote $\Gal(L^{S}/L_n) $.
			Note that if  $ w|v $ is a prime of $  L_\infty $ not above $ p $, then $ L_{\infty,w} $ is the unique unramified $ \bZ_p $-extension (\textit{cf.} for example \cite[Proposition 13.2]{Wash1}) of the local field $ L_v $.
			%
		Let $ \cL \in \{ K^{(l)}_{r}, \bK^{(l)}_{r}, L_{n} \} $ and $ \cG \in \{ G_{K^{(l)}_{r}}, G_{\bK^{(l)}_{r}}, G_n^{S}  \} $. 
		For any integer $ t \geq 0 $, there is a short exact sequence of $ \cG  $-modules 
		\[
		0 \longrightarrow
		E(\cL)/p^{t}E(\cL) \xrightarrow{\kappa^{p^{t}}_{\cL}} 
		H^{1}(\cG, E_{p^{t}}) \longrightarrow
		H^{1}(\cG, E)_{p^{t}} \longrightarrow
		0,
		\]
		where the map  $ \kappa^{p^{t}}_{\cL} $ is the Kummer map for $ E_{p^{t}} $ over $ \cL $. 
		Taking the direct limit of the above sequence we obtain the following
		\[
		0 \longrightarrow
		E(\cL) \otimes \bQ_p/\bZ_p \xrightarrow{\kappa^{p^{\infty}}_{\cL}} 
		H^{1}(\cG, E_{p^{\infty}}) \longrightarrow
		H^{1}(\cG, E)_{p^{\infty}} \longrightarrow
		0.
		\] 
		For $v \in S$, we define the local condition
		\begin{align}\label{def:Jv(Ln)}
			J_v(E_{p^{\infty}}/L_n) :=  \displaystyle\bigoplus_{w \mid v} H^1(L_{n,w},E)_{p^{\infty}}\cong
			\displaystyle\bigoplus_{w \mid v} \frac{H^1 (L_{n,w},E_{p^\infty})}{E(L_{n,w})  \otimes \bQ_p/\bZ_p},
		\end{align}
		where the isomorphism is due to the Kummer map (\textit{cf.} \cite[Section 1.6]{CoatesSujatha_book}). 
		The \textbf{classical $ p^{\infty} $-Selmer group} of the elliptic curve $ E $ over $ L_n $ {is defined} by the following sequence
		\begin{equation}\label{def:Selmer with J_n}
			0 \longrightarrow 
			\Sel(E_{p^{\infty}}/L_n) \rightarrow H^1(L^{S}/L_n, E_{p^\infty}) \xrightarrow{\lambda_n}  \displaystyle\bigoplus_{v \in S}J_v(E_{p^{\infty}}/L_n),
		\end{equation}
		where  $\lambda_n$ consists of restriction maps coming from Galois cohomology.
		By definition of the plus/minus norm groups, there are inclusions 
		\[
		E^{\pm}(L_{n,w}) \hooklongrightarrow E(L_{n,w}),\T{ and }  \widehat{E}^{\pm}(L_{n,w})  \hooklongrightarrow \widehat{E}(L_{n,w}),
		\]
		where  $ w $ is a prime in $ L_n $ above $ p $,  $ \fm_n $ is the maximal ideal of the ring of integers of $L_{n,w}(\mu_{p})  $.
			{By Lemma 3.4 of \cite{SujFil}, the above maps remain injective after applying the functor $ - \ot \ \bZ/p^{t}$ for any $ t \geq 1 $}
			\[
			\begin{tikzcd}
				E^{\pm}(L_{n,w})/ p^{t} E^{\pm}(L_{n,w}) \arrow[r, hook] & E(L_{n,w})/ p^{t} E(L_{n,w}),
			\end{tikzcd}
			\]
			where $ w $ is a prime in $ L_n $ above $ p $. 
			The Kummer map $ \kappa^{p^{t}}_{\cL} $ induces the following injective map
			\begin{equation}\label{signedkummermaps}
				\begin{tikzcd}
					\kappa^{\pm,p^{t}}_{L_{n,w}} \ : \ E^{\pm}(L_{n,w})/ p^{t} E^{\pm}(L_{n,w}) \arrow[r, hook] & H^{1}(L_{n,w}, E_{p^{t}}).
				\end{tikzcd}
			\end{equation}
			Similarly, for any $ n,t \geq 1  $, there is an injection
			\[
			\begin{tikzcd}
				\kappa^{\pm,p^{t}}_{L_{n,w}} \ : \ \widehat{E}^{\pm}(L_{n,w})/ p^{t} \widehat{E}^{\pm}(L_{n,w}) \arrow[r, hook] & H^{1}(L_{n,w}, \widehat{E}_{p^{t}}).
			\end{tikzcd}
			\]
			Refer to the maps $ \kappa^{\pm,p^{t}}_{L_{n,w}} $ as the \textbf{signed Kummer maps}  for $ E_{p^{t}} $ over $ L_{n,w} $.
		\subsection{Selmer group, signed Selmer groups, and fine signed residual Selmer groups}\label{sec:selmer-group-}
		Let $ E/L $ be an elliptic curve satisfying \hyperref[Hyp1]{Hyp 1} and let $ L_n/ L $ be as in \eqref{def:Ln}.
			{If $ w  $ belongs to $  \Sss(L_n) = S_p(L_n) $ then $ w| \fp $ or $ w| \bfp $, by \hyperref[Hyp1]{Hyp 1} part \hyperref[allss]{(iii)}.}
			Denote a prime in $ L_n $ which lies over $ \fp $ (\textit{resp.} $ \bfp $) by $ \fq $ (\textit{resp.} $ \bfq $).
			Similarly, if $  w  \in S_p(L_\infty) $ and $ w |\fp $ (\textit{resp.} $  w |\bfp $), denote $ w $ by $ \fq $ (\textit{resp.} $ \bfq $).
			Finally, if $ \fq  $ (\textit{resp.} $ \bfq  $) is in $ L_\infty  $, denote its restriction to $ L_n $ again by $ \fq $ (\textit{resp.} $ \bfq $).
			It should be clear from the context whether $ \fq  $ and $ \bfp $ represent primes in $ L_n $ or in $ L_\infty $.
		{Following \cite{SujFil}, for $ v \in S $ define the local cohomological groups }
		\begin{equation}\label{eq:K_v pm}
			^{\pm}\widetilde{K}_v(E_{p}/L_n):=
			\begin{dcases}
				\bigoplus_{w|\fl} H^{1}(L_{n,w}, E_{p}) & \text{if}\ v=\fl \in \Sbad ,\\
				\bigoplus_{\fq | \fp } H^{1}(L_{n,\fq}, E_{p}) / \Image \ \kappa^{\pm,p}_{L_{n,\fq}} & \text{if}\ v=\fp \in \Sss,\\
				\bigoplus_{\bfq | \bfp } H^{1}(L_{n,\bfq}, E_{p}) / \Image \ \kappa^{\pm,p}_{L_{n,\bfq}} & \text{if}\ v=\bfp \in \Sss.
			\end{dcases}
		\end{equation}
		Similarly
		\begin{equation}\label{eq:J_v pm}
			J_v^{\pm}(E_{p^{\infty}}/L_n):=
			\begin{dcases}
				\bigoplus_{w|\fl} H^{1}(L_{n,w}, E_{p^{\infty}}) & \text{if}\ v=\fl \in \Sbad, \\
				\bigoplus_{\fq | \fp } H^{1}(L_{n,\fq}, E_{p^{\infty}}) / \Image \ \kappa^{\pm,p^{\infty}}_{L_{n,\fq}} & \text{if}\ v=\fp \in \Sss,\\
				\bigoplus_{\bfq | \bfp } H^{1}(L_{n,\bfq}, E_{p^{\infty}}) / \Image \ \kappa^{\pm,p^{\infty}}_{L_{n,\bfq}} & \text{if}\ v=\bfp \in \Sss.
			\end{dcases}
		\end{equation}
		Here, when $ v \in \Sss $, the sign of $ ^{\pm}\widetilde{K}_v(E_{p}/L_n) $ (\textit{resp}. $ J_v^{\pm}(E_{p^{\infty}}/L_n) $) agrees with the choice of the sign of the Kummer map $ \kappa^{\pm,p} $ (\textit{resp}. $  \kappa^{\pm,p^{\infty}} $) in the direction of $ \fq|\fp $ or in the direction of $ \bfq|\bfp $.
		When $ v \in \Sbad $, the sign choice of the sign does not matter. 
		Similarly, when $ v \in \Sbad $ then $ J_v^{\pm}(E_{p^{\infty}}/L_n) $ coincides with $ J_v(E_{p^{\infty}}/L_n) $ for any choice of the sign. 
		The following definition is the analogue of \cite[Definition 3.6]{SujFil} over $ \bZ_p^{2} $-extensions.
		\begin{definition}\label{def:signed residual Selmer}
			Let $ n\geq 0 $
			For every intermediate field $L_n $ in the tower $ L_\infty/L $,  let $ ^{\pm}\widetilde{K}_v(E_{p}/L_n) $ be as \eqref{eq:K_v pm} for any $ v \in S $.
			Define the \textbf{fine signed residual Selmer group} $ \cR^{\pm/\pm}(E_p/L_n) $ of $ E_p $ over $ L_n  $ by:
			\begin{align*}
				\cR^{\pm/\pm} (E_p/L_n) := \ker \left( H^{1}(G^{S}_n,E_p) \to 
				\displaystyle\bigoplus_{v \in S}
				{^{\pm}\widetilde{K}_v(E_{p}/L_n)} \right).
			\end{align*}
			The choice of the first sign of $ \cR^{\pm/\pm} (E_p/L_n) $  is in accordance with the choice of $ ^{+}\widetilde{K}_{\fp}(E_{p}/L_n) $ or $ ^{-}\widetilde{K}_{\fp}(E_{p}/L_n) $ and the second sign is in accordance with the choice of $ ^{+}\widetilde{K}_{\bfp}(E_{p}/L_n) $ or $ ^{-}\widetilde{K}_{\bfp}(E_{p}/L_n)$.
		\end{definition}
		Here, we use the notation $ \cR^{\pm/\pm} (E_p/L_n) $ for convenience. 
			It unifies notation for any of the four possibilities: "$  \cR^{+/+} (E_p/L_n) $, or $ \cR^{+/-} (E_p/L_n) $, or $ \cR^{-/+} (E_p/L_n) $, or $ \cR^{-/-} (E_p/L_n) $".
			Similar notation is used throughout the article for simplicity.
		\begin{definition}\label{def:signed Selmer}
			For every $ n\geq 0$ let $ J_v^{\pm}(E_{p^{\infty}}/L_n) $ be as \eqref{eq:J_v pm} for any $ v \in S $.
			Define {the \textbf{signed Selmer groups} } $ \Sel^{\pm/\pm}(E_{p^{\infty}}/L_n) $ of $ E_{p^{\infty}} $ over the intermediate field $L_n $ in the tower $ L_\infty/L $ as follows 
			\begin{align*}
				\Sel^{\pm/\pm} (E_{p^{\infty}}/L_n) := \ker \left( H^{1}(G^{S}_n,E_{p^{\infty}}) \to 
				\displaystyle\bigoplus_{v \in S} J_v^{\pm}(E_{p^{\infty}}/L_n) \right).
			\end{align*}
			The first sign of $ \Sel^{\pm/\pm} (E_{p^{\infty}}/L_n) $  corresponds to the choice of $ J_\fp^{+}(E_{p^{\infty}}/L_n) $ or $ J_\fp^{-}(E_{p^{\infty}}/L_n) $ and 
			the second sign corresponds to the choice of $ J_\bfp^{+}(E_{p^{\infty}}/L_n) $ or $ J_\bfp^{-}(E_{p^{\infty}}/L_n) $.
		\end{definition}
		\begin{rem}\label{rem:adding-finite-set-of-primes-plus-minus}
			Suppose $ S\pr $ is a finite set of primes containing set $ S $.
			Then for any prime $ v \in S\pr \backslash S $, define the local condition $ J_v^{\pm}(E_{p^{\infty}}/L_n) $ or ${^{\pm}}\widetilde{K}_v(E_{p}/L_n)  $ the same way we did for the bad primes in $ \Sbad $.
			We note that adding a finite set of primes to the set $ S $ does not change these Selmer groups (\cf Chapter 1, Section 1.7 of \cite{SujathabookTIFR}).
		\end{rem}
		Let $ G_n := \Gal(L_n/L) $. 
		Then,  there is an exact sequence of $ \Lambda(G_n) $-modules (\textit{resp}. $ \Omega(G_n) $-modules):
		\begin{equation}\label{eq:signed Selmer with J_v pm over L_n}
			0 \longrightarrow 
			\Sel^{\pm/\pm} (E_{p^{\infty}}/L_n) \longrightarrow H^1(G_n^{S}, E_{p^{\infty}}) \longrightarrow \displaystyle\bigoplus_{v \in S} J_v^{\pm}(E_{p^{\infty}}/L_n)
		\end{equation}
		\textit{resp}. 
		\begin{equation}\label{eq:signed residual Selmer with J_v pm over L_n}
			0 \longrightarrow 
			\cR^{\pm/\pm} (E_p/L_n) \longrightarrow H^1(G_n^{S}, E_{p}) \longrightarrow  \displaystyle\bigoplus_{v \in S}
			{^{\pm}\widetilde{K}_v(E_{p}/L_n)}.
		\end{equation}
		By taking the direct limit over the intermediate field extensions $L_n$ of the exact sequences \eqref{def:Selmer with J_n}, \eqref{eq:signed Selmer with J_v pm over L_n}, and \eqref{eq:signed residual Selmer with J_v pm over L_n} we obtain
		\begin{align}
			0 \longrightarrow 
			\Sel(E_{p^{\infty}}/L_\infty) \longrightarrow H^1(G_{\infty}^{S}, E_{p^\infty}) \xrightarrow{\lambda_\infty}  \displaystyle\bigoplus_{v \in S}J_v(E_{p^{\infty}}/L_\infty),\label{eq:Selmer infty}\\
			0 \longrightarrow 
			\Sel^{\pm/\pm} (E_{p^{\infty}}/L_\infty) \longrightarrow H^1(G_{\infty}^{S}, E_{p^{\infty}}) \xrightarrow{\xi^{\pm/\pm}}  \displaystyle\bigoplus_{v \in S} J_v^{\pm}(E_{p^{\infty}}/L_\infty),\label{eq:signed Selmer infty with J_v pm}\\
			0 \longrightarrow
			\cR^{\pm/\pm} (E_p/L_\infty) \longrightarrow H^1(G_{\infty}^{S}, E_{p}) \xrightarrow{\xi_p^{\pm/\pm}}  \displaystyle\bigoplus_{v \in S} {^{\pm}\widetilde{K}_v(E_{p}/L_\infty)}. \label{eq:signed residual Selmer infty with k_v pm}
		\end{align}
		The above exact sequences   \eqref{eq:Selmer infty}, \eqref{eq:signed Selmer infty with J_v pm}, and \eqref{eq:signed residual Selmer infty with k_v pm} have a natural $ G $-action.
		This action gives these Selmer groups a module structure over Iwasawa algebras, which we now introduce.
		\begin{definition}\label{def:Iwasawa algebra}
			The \textbf{Iwasawa algebra} of $ \cG $, denoted by $ \Lambda(\cG) $ (\textit{resp}. $ \Omega(\cG) $), to be the completed group algebra of $ \cG $ over $ \bZ_p $ (\textit{resp.} $ \bF_p $).  
			That is,
			\[
			\Lambda(\cG):=\bZ_{p}[[\cG]]=\varprojlim_{\cN \subset \cG} \bZ_{p}[\cG/\cN], \  (\T{\textit{resp.} } \Omega(\cG):=\bF_{p}[[\cG]]=\Lambda(\cG)/p\Lambda(\cG)),
			\]
			where $\cN$ runs over open normal sub-groups of $\cG$.		
		\end{definition}
		
		In this article, we deal with $\cG \cong \bZ_p^n$ from some $ n  \in \{1 , 2\} $. 
		When the group $  \cG = G \cong \bZ_p^2$  then $ n = 2 $  and when we have $   \cG = \Gamma \cong \bZ_p$ then  $ n=1 $.
		The Iwasawa algebra $\Lambda(\cG)$ (\textit{resp}. $ \Omega(\cG) $) is isomorphic to the ring of formal power series $\bZ_p[[T_1,\cdots,T_n]]$ (\textit{resp}. $\bF_{p}[[T_1,\cdots,T_n]]$) with indeterminate variables $T_1,\cdots,T_n$. 
		Therefore $\Lambda(\cG)$ is a commutative, regular local ring of Krull dimension $n+1$.

		Note that \eqref{eq:Selmer infty} and \eqref{eq:signed Selmer infty with J_v pm} are exact sequences of $ \Lambda(G) $-modules and \eqref{eq:signed residual Selmer infty with k_v pm} is an exact sequence of $ \Omega(G) $-modules.
		Let $\mathfrak{X}(E_{p^{\infty}}/L_\infty)$ denote the Pontryagin dual of $\Sel(E_{p^{\infty}}/L_\infty) $.
		It is important to note here that $\mathfrak{X}(E_{p^{\infty}}/L_\infty)$ is conjectured to have positive rank, and therefore it is not a torsion $ \Lambda(G) $-module. 
		This is shown over the cyclotomic $\bZ_p$-extensions (\textit{cf.} \cite[Theorem 2.5]{CoatesSujatha_book}). 
		For a Galois extension with a pro-$ p $  Galois group without $ p $-torsion which contains the cyclotomic  $\bZ_p$-extension, this is conjectured to hold (\textit{cf.}  \cite{OchiVenjakov}).
		However, the Pontryagin dual of the signed Selmer group $\Sel^{\pm/\pm} (E_{p^{\infty}}/L_\infty)   $, which we denote by $\mathfrak{X}^{\pm/\pm}(E_{p^{\infty}}/L_\infty)$, is conjectured to be torsion as a $ \Lambda(G) $-module. For a more generalized version of this see Conjecture 4.11 of \cite{Lei-Lim}.
			We denote
			\[
			\mu_G^{\pm/\pm}(E_{p^{\infty}}/L_\infty):=\mu_G(\mathfrak{X}^{\pm/\pm}(E_{p^{\infty}}/L_\infty)),
			\] 
			and we call them the \textbf{signed $ \mu $-invariants}. 
			Suppose that $ M $ is a discrete cofinitely generated $\Lambda(\cG)$-module and let $ M\Pd $ denote its Pontryagin dual.
			Then, there is an isomorphism
				$ (M_p)\Pd \cong M\Pd / p M\Pd $
			and the  inequality
			\begin{align}\label{ineq:coranks}
				\T{corank}_{\Lambda(\cG)} (M) \leq \T{corank}_{\Omega(\cG)} (M_p).
			\end{align}
			{The inequality \eqref{ineq:coranks} becomes an equality exactly when the $\mu$-invariant of $M\Pd$ is zero. }

		
		\section{Local and global cohomology calculations}\label{sec:local and global cohomology calculations}
		Throughout this section, we assume that $ E/L $ is an elliptic curve satisfying \hyperref[Hyp1]{Hyp 1} and $ L_n/L $ is as defined in equation \eqref{def:Ln} for any $ n \geq 0 $.
		Let $ \mathcal{Y}^{\pm/\pm}(E_{p}/L_\infty) $ denote the Pontryagin dual of the fine signed residual Selmer group $ \cR^{\pm/\pm}(E_{p}/L_\infty) $ defined in Definition \ref{def:signed residual Selmer}.
		Our goal in this article is to show that if $ E_1/L $ and $ E_2/L $ are two elliptic curves satisfying \hyperref[Hyp1]{Hyp 1} and they are such that their residual Galois representations $ (E_1)_{p} $ and $(E_2)_{p}  $ are isomorphic, then 
		\[
		\mu_G^{\pm/\pm}((E_{1})_{p^{\infty}}/L_\infty) = 0 \iff 
		\mu_G^{\pm/\pm}((E_{2})_{p^{\infty}}/L_\infty) = 0
		\]
		given that $ \mathcal{Y}^{\pm/\pm}((E_{j})_{p}/L_\infty) $  is a torsion $ \Omega(G) $-module for $ j \in \{1,2\} $.
		\subsection{Comparison of the local and global cohomology groups}
		Let $ n\geq 0 $ and let $ L_n/L $ be as in  \eqref{def:Ln}.
		{
				For convenience, given a local field  $ \cK $, let $ G_{\cK} $ denote the Galois group of the extension $ \overline{\cK}/\cK $ where $ \overline{\cK} $ denotes the separable closure of $ \cK $. 
		}
		Suppose $ \cG \in \{ G_n^{S}, G_{L_{n,w}} \} $ where $ w \in S(L_n) $.
		There is a short exact sequence
		\begin{equation}\label{eq:step-to-define-psi}
			0 \longrightarrow 
			E(\cL)_{p^{\infty}}/p E(\cL)_{p^{\infty}} \longrightarrow 
			H^{1}(\cG, E_p) \xrightarrow{\psi_{\cG,n}}
			H^{1}(\cG, E_{p^{\infty}})_p  \longrightarrow
			0.
		\end{equation}
		Recall that $ L_{n,w}= K^{(l)}_r $ for some $ l,r \geq 0 $. Equation \eqref{eq:Ep=0} tells us that  $ 	E(L_{n,w})_p = E(L_{n})_p =\{ 0 \}$.
		This means that the first term in the short exact sequence \eqref{eq:step-to-define-psi} vanishes. 
		Therefore for any $ n \geq 0 $,
		the map $ \psi_{\cG,n}  $ is an isomorphism.
		It remains to investigate $ \psi_{\cG,n}  $ when $  \cG= G_{L_{n,w}} $ for some  $ w \in \Sbad(L_n)  $. 
		Note that 
		$ \ker(\psi_{\cG,n})= E(L_{n,w})_{p^\infty}/p E(L_{n,w})_{p^\infty} $
		has a finite $ \bF_p $-dimension 
		and by \cite[Theorem 3.2]{Washington2}
		\[
		\dim_{\bF_{p}}(\ker(\psi_{\cG,n}))=\dim_{\bF_{p}}(E(L_{n,w})_{p}) \leq \dim_{\bF_{p}}(E(\overline{L_{w}})_{p}) =2.
		\] 
		Passing the exact sequence \eqref{eq:step-to-define-psi} to direct limit, define the surjective map
		\begin{equation}\label{eq:define-psi-infty}
			\psi_{\cG,\infty} : H^{1}(\cG, E_p) \twoheadrightarrow H^{1}(\cG,E_{p^{\infty}})_p,
		\end{equation} 
		where $ \cG \in \{ G_\infty^{S}, G_{L_{\infty,w}} \} $ for $ w \in S(L_\infty) $ and 
		$ G_\infty^{S} := \Gal(L_\infty^{S}/L_\infty) $. 
		\begin{prop}\label{prop:psi-infty}
			Let $ \cG \in \{ G_\infty^{S}, G_{L_{\infty,w}} \} $ for $ w \in S(L_\infty) $ and let $ \psi_{\cG,\infty} $ be the map described in \eqref{eq:define-psi-infty}.
			\begin{enumerate}
				\item\label{prop:psi-infty-1} If $ \cG \in \{ G_\infty^{S}, G_{L_{\infty,w}} \} $ for $ w \in S_p(L_\infty) $, then $ \psi_{\cG,\infty} $ is an isomorphism.
				\item\label{prop:psi-infty-2} If $  \cG= G_{L_{\infty,w}} $ for  $ w \in \Sbad(L_\infty)  $, then
				$ \dim_{\bF_{p}}(\ker(\psi_{\cG,\infty})) \leq 2 $.
			\end{enumerate}
		\end{prop}
		\begin{proof}
			The proof is similar to Proposition 4.1 in \cite{SujFil}.  
		\end{proof}
			Suppose $ w \in S(L_\infty)  $ and assume $ w |\fp $ (\textit{resp.} $ w|\bfp $). 
			{Denote $ w$  by $ \fq $ (\textit{resp.} by $ \bfq $) and the isomorphism  $ \psi_{\cG,\infty} $ where $ \cG = \Gal(\overline{L_{\infty,\fq}}/L_{\infty,\fq}) $ by $ \psi_{\fq,\infty} $ (\textit{resp.} $ \psi_{\bfq,\infty} $).
			}
			{Finally, when $  w \in \Sbad(L_\infty)$ then $L_{\infty,w} = L_{\cyc,w}  $ and $ \cG = \Gal(\overline{L_{\infty,w}}/L_{\infty,w}) $.
				Denote the surjective map $ \psi_{\cG,\infty} $ by $  \psi_{w,\infty}$.}
		\begin{cor}\label{cor:psi-infty-fq}
			Let $  \fq $  (\textit{resp.} $ \bfq $) be a prime in $ S_p(L_\infty) $ above $ \fp $ (\textit{resp.} $ \bfp $). Then,	 the isomorphism $ \psi_{\fq,\infty} $ (\textit{resp.} $ \psi_{\bfq,\infty}$) induces an isomorphism 
			\begin{equation*}\label{eq:define-psi-infty-fq}
				\psi_{\fq,\infty}^{\pm} : 
				H^{1}(L_{\infty,\fq}, E_p)/ \Image \left(\kappa^{\pm,p}_{L_{\infty,\fq}}\right) \xrightarrow{\cong} \left(H^{1}(L_{\infty,\fq},E_{p^{\infty}})/ \Image \left(\kappa^{\pm,p^{\infty}}_{L_{\infty,\fq}}\right)\right)_p
			\end{equation*}
			\textit{resp.}
			\begin{equation*}\label{eq:define-psi-infty-bfq}
				\psi_{\bfq,\infty}^{\pm} : 
				H^{1}(L_{\infty,\bfq}, E_p)/ \Image \left(\kappa^{\pm,p}_{L_{\infty,\bfq}}\right) \xrightarrow{\cong} \left(H^{1}(L_{\infty,\bfq},E_{p^{\infty}})/ \Image \left(\kappa^{\pm,p^{\infty}}_{L_{\infty,\bfq}}\right)\right)_p.
			\end{equation*}
		\end{cor}
		\begin{proof}
			The proof is similar to \cite[Proposition 4.1, part d)]{SujFil} .
		\end{proof}

		\subsection{Fine signed residual Selmer group and residual representations}\label{sec:fine-signed-and-residual-representations}
		{
			This section explains why the fine signed residual Selmer group only depends on the residual Galois representation $ E_p $.
			To do this, we show that the local conditions $  ^{\pm}\widetilde{K}_v(E_{p}/L_\infty)$, which define the fine signed residual Selmer group (\textit{cf.} Definition \ref{def:signed residual Selmer}), only depend on the residual Galois representation (\textit{cf.} Proposition \ref{prop:fine selmer only depends on residual reps}). 
			Suppose $ \cL$ is an algebraic extension of $ \bQ_p $.
			By Lemma 3.3 of \cite{SujFil}, there exists an exact sequence
			\[
			\begin{tikzcd}\label{exact:lemma3.3-suj-fil}
				0 \arrow[r] & \widehat{E}(\cL) \arrow[r] & E(\cL) \arrow[r] & D \arrow[r] & 0,
			\end{tikzcd}
			\]
			where $D  $ is a finite group of order prime-to-$ p $. 
			In particular, for any $ n \geq 0 $ and any prime $ w \in \Sss(L_n)=S_p(L_n) $,
			\[
			\widehat{E}^{\pm}(L_{n,w}) \ot \bQ_p/\bZ_p \cong E^{\pm}(L_{n,w}) \ot \bQ_p/\bZ_p.
			\]
			If $ w $ is a prime in $ \Sss(L_\infty) $, then
			$ \widehat{E}^{\pm}(L_{\infty,w}) \ot \bQ_p/\bZ_p \cong E^{\pm}(L_{\infty,w}) \ot \bQ_p/\bZ_p$.
			Putting this together with equation \eqref{signedkummermaps}, for all $ n \geq 0 $
			\begin{align}\label{eq:imkummermapiso}
				\Image \ \kappa^{\pm,p^{\infty}}_{L_{n,w}} (\Epinf^{\pm}(L_{n,w})) \cong 
				\Image \ \kappa^{\pm,p^{\infty}}_{L_{n,w}} (\wEpinf^{\pm}(L_{n,w})).
			\end{align}
			Furthermore, the exact sequence \eqref{exact:lemma3.3-suj-fil} induces an isomorphism for any $ n \geq 0 $ and any prime $ w \in \Sss(L_n)$,
				$ H^{1}(L_{n,w},E_{p^{\infty}}) \cong H^{1}(L_{n,w}\widehat{E}_{p^{\infty}}) $.
			The isomorphisms in the above equation and equation \eqref{eq:imkummermapiso} are compatible for all $ n \geq 0 $ and they induce the isomorphism
			\begin{align}\label{eq:H1ofEandEhatiso}
				H^{1}(L_{n,w},E_{p^{\infty}})/ \Image \ \kappa^{\pm,p^{\infty}}_{L_{n,w}}
				\cong 
				H^{1}(L_{n,w},\widehat{E}_{p^{\infty}})/\Image \ \kappa^{\pm,p^{\infty}}_{L_{n,w}}.
			\end{align}
			{Note that for $ n  \geq 0 $ and  $  w \in \Sss(L_n)  $, then $L_{n,w} =  K^{(l)}_{r} $ for some $ l,r \geq 0 $.
				Let $ \cO $ denote the ring of integers of $ K^{(l)}_{0} $.	
				By Honda theory (\textit{cf.} \cite{Honda1970}), there exists a formal group $ \cF_{\T{ss}} $ in $ \cO[[X]] $, known as the supersingular formal group of Honda type $ t^{2}+p $ which is an  Eisenstein polynomial. 
				Further, part (ii) and (iv) of \hyperref[Hyp1]{Hyp 1} implies that the formal group  $ \widehat{E} $
				is also of Honda type $ t^{2}+p  $ (\cf section 8.2 of \cite{Kobayashi}). 
				Therefore, there exists an
				$ \cO $-isomorphism 
				\begin{align}\label{eq:expologiso}
					\exp_{\widehat{E}} \circ \log_{\cF_{\T{ss}}} \ : \ \cF_{\T{ss}}(\fm) \to \widehat{E}(\fm)
				\end{align}
				where $ \fm $	is the maximal ideal of the valuation ring $ \cO $ of the local field $ L_{n,w} $. 
				Since the Honda type $ t^{2}+p $ is independent of the choice of the elliptic curve, the supersingular formal group  $ \cF_{\T{ss}} $, and hence the formal group $ \widehat{E}  $ by the above isomorphism, are  independent of the choice of elliptic curve $ E/L $, assuming $ E/L $ satisfies \hyperref[Hyp1]{Hyp 1}. }
			
			For simplicity we denote $ \cF_{\T{ss}}(\fm) $ by $ \cF_{\T{ss}}(L_{n,w}) $ and define the plus minus norm groups $ \cF_{\T{ss}}^{\pm}(L_{n,w})  \subseteq \cF_{\T{ss}}(L_{n,w})  $ using the isomorphism \eqref{eq:expologiso}
			\[
			\cF_{\T{ss}}^{\pm}(L_{n,w}):=\log_{\cF_{\T{ss}}}  \circ \exp_{\widehat{E}}(\widehat{E}^{\pm}(L_{n,w})).
			\]
			Moreover, the signed Kummer maps, defined for the formal groups in equation \eqref{signedkummermaps}, can be defined for the supersingular formal group $ \cF_{\T{ss}}(L_{n,w}) $
			\begin{small}
				\[
				\begin{tikzcd}
					0 \rightarrow 
					\cF_{\T{ss}}^{\pm}(L_{n,w})\otimes \bQ_p/\bZ_p \xrightarrow{\kappa^{\pm,p^{\infty}}_{L_{n,w}}}
					H^{1}(L_{n,w}, (\cF_{\T{ss}})_{p^{\infty}})
					\rightarrow 
					H^{1}(L_{n,w}, (\cF_{\T{ss}})_{p^{\infty}}) / \Image \ \kappa^{\pm,p^{\infty}}_{L_{n,w}}\rightarrow 
					0.
				\end{tikzcd}
				\]
			\end{small}
			
			The isomorphism \eqref{eq:expologiso} is $ \cO_{K^{(l)}_{0}} $-linear and thus it commutes with $ G_{L_{n,w}} $-action. 
			Hence, there is an isomorphism
			\begin{align}\label{eq:H1ofEhatandFssiso}
				H^{1}(L_{n,w},(\cF_{\T{ss}})_{p^{\infty}})/\Image \ \kappa^{\pm,p^{\infty}}_{L_{n,w}}
				\cong 
				H^{1}(L_{n,w},\widehat{E}_{p^{\infty}})/\Image \ \kappa^{\pm,p^{\infty}}_{L_{n,w}}.
			\end{align}
			Combining \eqref{eq:H1ofEandEhatiso} and \eqref{eq:H1ofEhatandFssiso}, for $ w \in \Sss(L_n)  $
			\begin{align}\label{eq:H1ofEandFssiso}
				H^{1}(L_{n,w},(\cF_{\T{ss}})_{p^{\infty}})/\Image \ \kappa^{\pm,p^{\infty}}_{L_{n,w}}
				\cong 
				H^{1}(L_{n,w},E_{p^{\infty}})/ \Image \ \kappa^{\pm,p^{\infty}}_{L_{n,w}}.
			\end{align}
			Since the left-hand side is independent of $ E $, so is the right-hand side. 
			This implies that $ J_\fp^{\pm}(E_{p^{\infty}}/L_n) $ and $J_\bfp^{\pm}(E_{p^{\infty}}/L_n) $ are independent of the choice of  $ E/L $ for any $ n \geq 0 $. 
			Passing to the direct limit we have the following isomorphism
			\begin{align}\label{eq:H1ofEandFssiso-infty}
				H^{1}(L_{\infty,w},(\cF_{\T{ss}})_{p^{\infty}})/\Image \ \kappa^{\pm,p^{\infty}}_{L_{\infty,w}}
				\cong 
				H^{1}(L_{\infty,w},E_{p^{\infty}})/ \Image \ \kappa^{\pm,p^{\infty}}_{L_{\infty,w}},
			\end{align}
			where $ w \in \Sss(L_\infty)=S_p(L_\infty) $ which implies that the same is true for that the local conditions  $ J_\fp^{\pm}(E_{p^{\infty}}/L_\infty) $ and $J_\bfp^{\pm}(E_{p^{\infty}}/L_\infty) $. 
		}
		%
		%
			Suppose $ E_1/L $ and $ E_2/L $ are two elliptic curves satisfying \hyperref[Hyp1]{Hyp 1}.
			Let $ j \in \{1, 2\} $ and let 
			$ 	S_j = \{\fp, \bfp \} \cup \Sbad_j  $
			{where $ \Sbad_j $ is the set of primes where $ E_j $ has bad reduction.}
			From now on, we shall now enlarge the set $ S $ by declaring
			\[
			S : = S_1 \cup S_2 = \{\fp, \bfp \} \cup \Sbad_1 \cup \Sbad_2.
			\]
			Recall from Remark \ref{rem:adding-finite-set-of-primes-plus-minus} that adding a finite set of primes to the set $ S_j $ does not change $ \cR^{\pm/\pm} ((E_j)_p/L_\infty)  $ and $ \Sel^{\pm/\pm} ((E_{j})_{p^{\infty}}/L_\infty) $. 
		
		\begin{prop}\label{prop:fine selmer only depends on residual reps}
			Let  $E_1/L $ and $ E_2/L $ be two elliptic curves that satisfy \hyperref[Hyp1]{Hyp 1}.
			Moreover, suppose  $  (E_1)_p \cong (E_2)_p $ as $ \Gal(\overline{L}/L) $-modules. 
			Then, there is an isomorphism
			\[
			\cR^{\pm/\pm} ((E_1)_p/L_\infty) \cong \cR^{\pm/\pm} ((E_2)_p/L_\infty).
			\]
		\end{prop}
		\begin{proof}
			{Let $ j \in \{1, 2\} $ and $ S $ be as above.}
			Suppose $ v = \fp $. Then, for any $ \fq | \fp  $ in $ L_{\infty} $ by Corollary \ref{cor:psi-infty-fq} and isomorphism \eqref{eq:H1ofEandFssiso-infty}
			\begin{align*}
				H^{1}(L_{\infty,\fq}, (E_j)_p)/ \Image \left(\kappa^{\pm,p}_{L_{\infty,\fq}}\right)&\cong \left(H^{1}(L_{\infty,\fq},(E_j)_{p^{\infty}})/ \Image \left(\kappa^{\pm,p^{\infty}}_{L_{\infty,\fq}}\right)\right) \\
				&\cong 
				\left(H^{1}(L_{\infty,w},(\cF_{\T{ss}})_{p^{\infty}})/\Image \ \kappa^{\pm,p^{\infty}}_{L_{\infty,w}}\right)_p.
			\end{align*}
			The right-hand side is independent of the choice $ j \in \{ 1, 2\} $. 
			Therefore,
			\[
			H^{1}(L_{\infty,\fq}, (E_1)_p)/ \Image \left(\kappa^{\pm,p}_{L_{\infty,\fq}}\right)\cong
			H^{1}(L_{\infty,\fq}, (E_2)_p)/ \Image \left(\kappa^{\pm,p}_{L_{\infty,\fq}}\right)
			\]
			which implies that 
			$
			{^{\pm}\widetilde{K}_\fp((E_1)_{p}/L_\infty)} \cong {^{\pm}\widetilde{K}_\fp((E_2)_{p}/L_\infty)}.
			$
			Similar argument works for when $ v = \bfp  $.
			Now let $ v \in S \backslash \{\fp, \bfp \}  $.
			Since $(E_1)_p \cong (E_2)_p $, it follows that for any prime $ w |v $
			\[
			H^{1}(L_{n,w}, (E_1)_{p}) \cong H^{1}(L_{n,w}, (E_2)_{p})
			\]
			and thus
			$
			^{\pm}\widetilde{K}_v((E_1)_{p}/L_n) \cong ^{\pm}\widetilde{K}_v((E_2)_{p}/L_n).
			$
			This means that for any $v \in S   $ the local terms match, and therefore we have the result.
		\end{proof}
		
		{
			\begin{prop}\label{prop:fine and p-primary of selmer have the same corank}
				Let $ E/L $ be an elliptic curve satisfying \hyperref[Hyp1]{Hyp 1}. Then, there exists an injective map 
				\begin{align*}
					\varphi^{\pm/\pm}: \cR^{\pm/\pm} (E_p/L_\infty) \hooklongrightarrow \Sel^{\pm/\pm} (E_{p^{\infty}}/L_\infty)_p
				\end{align*}
				such that the $\mathrm{coker}(\varphi^{\pm/\pm} )  $ is a cotorsion $\Omega(G) $-module. Therefore, the following equality holds
				\begin{align*}
					\mathrm{corank}_{\Omega(G)}(\cR^{\pm/\pm} (E_p/L_\infty))= \mathrm{corank}_{\Omega(G)}(\Sel^{\pm/\pm} (E_{p^{\infty}}/L_\infty)_p).
				\end{align*}
			\end{prop}
			{
				\begin{proof}
					Let $ v \in S $ and suppose $ w |v $ is a prime in $ L_\infty $. 
					Set
					\begin{align*}
						{\varphi_{w}^{\pm}}:=
						\begin{dcases}
							\psi_{w,\infty}
							& \text{if}\ w| \fl \in \Sbad,\\
							\psi^{\pm}_{\fq,\infty}
							& \text{if}\ w=\fq| \fp,\\
							\psi^{\pm}_{\bfq,\infty}
							& \text{if}\ w=\bfq| \bfp.
						\end{dcases}
					\end{align*}
					The maps  $ \psi^{\pm}_{\fq,\infty} $ and $ \psi^{\pm}_{\bfq,\infty} $ are isomorphisms by Proposition \ref{prop:psi-infty}-\hyperref[prop:psi-infty-1]{(1)} and the $ \bF_p $-dimension of $ \ker(\psi_{w,\infty})$ is less than or equal to $ 2 $ by Proposition \ref{prop:psi-infty}-\hyperref[prop:psi-infty-2]{(2)}.
					Define the map
					\begin{align}\label{eq:varphi{v}{pm}}
						{\varphi_{v}^{\pm}} : {^{\pm}\widetilde{K}_v(E_{p}/L_\infty)} &\to \left(J_v^{\pm}(E_{p^{\infty}}/L_\infty)\right)_p\\
						\varphi_{v}^{\pm} &: = \bigoplus_{w|v} \varphi_{w}^{\pm}. \nonumber
					\end{align}
					Therefore,  $ \varphi_{v}^{\pm} $ is an isomorphism  when $ v \in \{\fp,\bfp \} $. 
					Define the map $ \varphi^{\pm/\pm} $ using the following commutative diagram (\textit{cf.} \cite[Corollary 4.5]{SujFil})
					\begin{small}
						\begin{equation}\label{eq:commutative-diagram-varphi{pm/pm}}
							\begin{tikzcd}
								0 \arrow[r] & \cR^{\pm/\pm} (E_p/L_\infty) \arrow[r] \arrow[d, "\varphi^{\pm/\pm}"']          & H^1(G_{\infty}^{S}, E_{p}) \arrow[r, "\xi_p^{\pm/\pm}"] \arrow[d, "\psi_{G_{\infty}^{S},\infty}"', "\cong"]    & \displaystyle\bigoplus_{v \in S} {^{\pm}\widetilde{K}_v(E_{p}/L_\infty)} \arrow[d, "{\ds\bigoplus_{v \in S}\varphi_{v}^{\pm}}"] \\
								0 \arrow[r] & \Sel^{\pm/\pm} (E_{p^{\infty}}/L_\infty)_p \arrow[r] & H^1(G_{\infty}^{S}, E_{p^{\infty}})_p \arrow[r, "\xi^{\pm/\pm}"'] & \displaystyle\bigoplus_{v \in S} \left(J_v^{\pm}(E_{p^{\infty}}/L_\infty)\right)_p.              	
							\end{tikzcd}
						\end{equation}	
					\end{small}
					
					The middle vertical map is an isomorphism by Proposition \ref{prop:psi-infty}-\hyperref[prop:psi-infty-1]{(1)}.  
					This implies that  $ \varphi^{\pm/\pm} $ is injective and 
					\begin{align*}
						\mathrm{coker}(\varphi^{\pm/\pm} ) \subseteq \ker(\bigoplus_{v \in S}\varphi_{v}^{\pm}) 
						= \bigoplus_{v \in \Sbad}\ker(\varphi_{v}^{\pm}).
					\end{align*}
					Suppose $ v \in \Sbad $ and fix a prime $ w|v $ in $ L_\infty $.
					Then we have
					\begin{align*}
						\ker(\varphi_{v}^{\pm}) = \mathrm{Ind}^{G}_{G_{w}}\ker(\psi_{w,\infty})
						\cong \ker(\psi_{w,\infty}) \widehat{\otimes}_{\Omega(G_{w}) } \Omega(G),
					\end{align*}
					where $ G_w = \Gal(L_{\cyc,w},L_v) \cong \bZ_p $ is the decomposition group
					of $ G $ at the prime $ w $ and $-\widehat{\otimes}_{\Omega(G_{w}) } -$ denotes the completed tensor product over $\Omega(G_{w})$.
					Recall that the map $ \psi_{w,\infty} $
					is surjective with finite kernel by Proposition \ref{prop:psi-infty}-\hyperref[prop:psi-infty-2]{(2)}.Thus $ \ker(\psi_{w,\infty}) $ is a cotorsion $\Omega(G_{w}) $-module. 
					As $\Omega(G)  $ is a flat $ \Omega(G_{w}) $-module (\textit{cf.} \cite[Lemma 3.3]{Schneider2010}), we have $\bigoplus_{v \in \Sbad}\ker(\varphi_{v}^{\pm}) $ is cotorsion as a $\Omega(G) $-module. 
				\end{proof}
			}		
			\section{Signed Selmer and fine residual Selmer groups as Iwasawa modules}\label{sec:Signed Selmer and fine residual Selmer groups as Iwasawa modules}

			Here, we prove our main theorem.
			Throughout this section, we assume that $ E/L $ is an elliptic curve satisfying the assumptions of \hyperref[Hyp1]{Hyp 1}.
			\begin{conjecture}(\cite[Conjecture 4.11]{Lei-Lim})\label{conj:X-infty-torsion}
				For any choice of the signs, 
				\newline $ \mathfrak{X}^{\pm/\pm}(E_{p^{\infty}}/L_\infty) $ is a torsion $ \Lambda(G) $-module.
			\end{conjecture}
				If $ E$ is defined over $ \bQ $, then Lei and Sprung showed in \cite{Lei2020} that the above conjecture holds (\textit{cf.} proof of \cite[Theorem 4.4]{Lei2020}).
			
			\begin{enumerate}[{\textbf{Hyp 2 $ ^{\pm/\pm}  $}} :]\label{Hyp2}
				\item $ \mathfrak{X}^{\pm/\pm}(E_{p^{\infty}}/L_\infty) $ is a torsion $ \Lambda(G) $-module.
			\end{enumerate}
			
			Note that Conjecture \ref{conj:X-infty-torsion} implies \hyperref[Hyp2]{Hyp 2$ ^{\pm/\pm} $} for all choices of the signs.
			In \cite{SujFil}, to prove their main results, the authors assume (\textit{cf.} \cite[Hyp 2]{SujFil}):
			\begin{enumerate}[{\textbf{Hyp 2 $ ^{\pm/\pm} $(cyc)}} :]\label{Hyp2(cyc)}
				\item $ \mathfrak{X}^{\pm/\pm}(E_{p^{\infty}}/L_\cyc)  $ is a torsion
				$ \Lambda(\Gamma) $-module. 
			\end{enumerate}
			Here, the $ \Lambda(\Gamma) $-module $ \mathfrak{X}^{\pm/\pm}(E_{p^{\infty}}/L_\cyc)  $ denotes the Pontryagin dual of the Selmer groups $ \Sel^{\pm/\pm}(E_{p^{\infty}}/L_\cyc)  $.  
			We show later that \hyperref[Hyp2(cyc)]{Hyp 2$ ^{\pm/\pm} $(cyc)} implies \hyperref[Hyp2]{Hyp 2$ ^{\pm/\pm} $} (\textit{cf.} Proposition \ref{pro:Hyp 2 (cyc) implies Hyp 2}).
			{If $ \Sel(E_{p^{\infty}}/L) = \Sel^{\pm/\pm}(E_{p^{\infty}}/L)  $ is finite then \hyperref[Hyp2(cyc)]{Hyp 2$ ^{\pm/\pm} $(cyc)} is automatically satisfied (\textit{cf.} Remark 4.5 of \cite{Lei-Sujatha}).
			}
			\begin{prop}\label{pro:Hyp 2 equivalence}
				\hyperref[Hyp2(cyc)]{Hyp 2$ ^{\pm/\pm} $(cyc)} (\textit{resp.} \hyperref[Hyp2]{Hyp 2$ ^{\pm/\pm} $}) holds if and only if
				\begin{enumerate}
					\item $ H^{2}(G_{\cyc}^{S},E_{p^{\infty}})$ (\textit{resp.} $H^{2}(G_{\infty}^{S},E_{p^{\infty}})$) vanishes; and
					\item The map $ \xi_{\cyc}^{\pm/\pm}  $ in the exact sequence (\textit{resp.} the map $ \xi^{\pm/\pm}  $ defined in \ref{eq:signed Selmer infty with J_v pm})
					\begin{align}\label{eq:short-exact-selmer-cyc}
						0 \rightarrow 
						\Sel^{\pm/\pm} (E_{p^{\infty}}/L_\cyc) \rightarrow H^1(G_{\cyc}^{S}, E_{p^{\infty}}) \xrightarrow{\xi_{\cyc}^{\pm/\pm}}  \displaystyle\bigoplus_{v \in S} J_v^{\pm}(E_{p^{\infty}}/L_\cyc)
					\end{align}
					is surjective.
				\end{enumerate}
			\end{prop}
			\begin{proof}
				{See Proposition 4.4 of \cite{Lei-Sujatha} (\textit{resp.} Proposition 4.12 of \cite{Lei-Lim}).}
			\end{proof}
			\subsection{Cassels--Poitou--Tate exact sequence for fine signed residual Selmer group}\label{Cassels--Poitou--Tate exact sequence for signed fine residual Selmer group}
			Here, we aim to produce an analogous statement to Remark \ref{pro:Hyp 2 equivalence} for  signed fine residual Selmer groups.
			More specifically, we would like to show that the Pontryagin dual of the signed fine residual Selmer group, $ \mathcal{Y}^{\pm/\pm}(E_{p}/L_\infty) $, is a torsion $ \Omega(G) $-module exactly when the exact sequence \eqref{eq:signed residual Selmer infty with k_v pm} is short exact.
			In other words, exactly when the map $ \xi_{p}^{\pm/\pm} $ is surjective.
			Our strategy is the same as \cite[Section 4.2]{SujFil}. We do this by studying the exact sequence \eqref{eq:signed residual Selmer infty with k_v pm}
			using Cassels--Poitou--Tate exact sequence (\textit{cf.} \cite[Theorem 1.5]{CoatesSujatha_book} for details).
			We begin by defining some new modules involved in the sequence.
			Let $ n \geq 0 $ and $ w \in S(L_n) $. Define
			\begin{align}\label{def:cW_w(E_p/L_n)}
				\cW^{\pm}_w(E_{p}/L_n):=
				\begin{dcases}
					0 & \text{if}\ w \in \Sbad(L_n), \\
					\Image \ \kappa^{\pm,p}_{L_{n,\fq}} & \text{if}\ w=\fq|\fp,\\
					\Image \ \kappa^{\pm,p}_{L_{n,\bfq}} & \text{if}\ w=\bfq|\bfp,
				\end{dcases}
			\end{align}
			where the choice of the sign, as usual, depends on the sign of the Kummer map $ \kappa^{\pm,p} $ in the direction of primes $ \fp $ and $  \bfp$.
			For $ v \in S $ (\textit{cf.} equation \eqref{eq:K_v pm})
			\[
			^{\pm}\widetilde{K}_v(E_{p}/L_n)=\bigoplus_{w|v} H^{1}(L_{n,w}, E_{p}) / \cW^{\pm}_w(E_{p}/L_n).
			\]
			Let $ {\cW^{\pm}_w(E_{p}/L_n)}^{\bot} \subset H^{1}(L_{n,w}, E_{p})  $ denote the orthogonal complement of the group $ {\cW^{\pm}_w(E_{p}/L_n)} $ with respect to the Tate local duality (\textit{cf.} \cite[Section 1.1]{CoatesSujatha_book}).
			Define $ \fR^{\pm/\pm} (E_p/L_n)  $, a $ \Omega(G_n) $-module, using the exact sequence 
			\begin{align*}
				0 \rightarrow \fR^{\pm/\pm} (E_p/L_n)  \rightarrow H^1(G_{n}^{S}, E_{p}) \rightarrow  \bigoplus_{w \in S(L_n)} H^{1}(L_{n,w}, E_{p}) / {\cW^{\pm}_w(E_{p}/L_n)}^{\bot}.
			\end{align*}
			{For any $ n \geq 0 $, the Cassels--Poitou--Tate exact sequence gives}
			\begin{align}\label{eq: fine selmer Cassels--Poitou--Tate exact sequence}
				0 \rightarrow \cR^{\pm/\pm} (E_p/L_n) \rightarrow H^1(G_{n}^{S}, E_{p})  \rightarrow \displaystyle\bigoplus_{w \in S(L_n)} H^{1}(L_{n,w}, E_{p}) / \cW^{\pm}_w(E_{p}/L_n) \\ \nonumber
				\rightarrow   \fR^{\pm/\pm} (E_p/L_n)\Pd \rightarrow H^2(G_{n}^{S}, E_{p}) \rightarrow \displaystyle\bigoplus_{w \in S(L_n)}  H^{2}(L_{n,w}, E_{p})  \rightarrow 0.
			\end{align}
			The final zero is due to equation \eqref{eq:Ep=0} (\textit{cf.} \cite[Theorem 1.5]{CoatesSujatha_book}).
			Define 
			\begin{align}
				\mathscr{R}^{\pm/\pm}(E_p/ L_{\infty}) := \varprojlim_{\mathrm{cores}} \fR^{\pm/\pm} (E_p/L_n) \subset H^{1}_{\mathrm{Iw}}(L,E_p) \label{eq:Iwasawacohomology}
			\end{align}
			where the Iwasawa cohomology module
				$ H^{1}_{\mathrm{Iw}}(L,E_p) := \varprojlim_{\mathrm{cores}} H^{1}(G_n^{S}, E_p) $.
			{Since  $ H^{0}(G_{n}^{S}, E_p) $ vanishes by equation \eqref{eq:Ep=0}, using Jannsen’s spectral sequence (\textit{cf.} \cite[Corollary 13]{Jannsen2014}), there is an isomorphism}
			\begin{align*}
				H^{1}_{\mathrm{Iw}}(L,E_p) \cong \Hom_{\Omega(G)}\left(\Hom_{\bF_p}(H^{1}(G^{S}_{\infty}),\bF_p),\Omega(G) \right).
			\end{align*}
			The left-hand side is a torsion-free $\Omega(G)$-module and therefore so is the right-hand side. 
			In particular,  $\mathscr{R}^{\pm/\pm}(E_p/ L_{\infty})  $ is also torsion-free being a  $\Omega(G)$-submodule of $ 	H^{1}_{\mathrm{Iw}}(L,E_p) $.
			Moreover (\textit{cf.} \cite[Lemma 4.6.]{SujFil})
			\begin{small}
				\begin{align*}
					\mathscr{R}^{\pm/\pm}(E_p/ L_{\infty})\Pd = \Hom(\varprojlim_{\mathrm{cores}} \fR^{\pm/\pm} (E_p/L_n), \bQ_p/\bZ_p)
					= \varinjlim_{\mathrm{cores}\Pd} \fR^{\pm/\pm} (E_p/L_n)\Pd.
				\end{align*}
			\end{small}
			
			Taking the direct limit of the exact sequence \eqref{eq: fine selmer Cassels--Poitou--Tate exact sequence} yields
			\begin{align}
				0\rightarrow\cR^{\pm/\pm} (E_p/L_\infty) \rightarrow H^1(G_{\infty}^{S}, E_{p}) \xrightarrow{\xi_p^{\pm/\pm}} \displaystyle\bigoplus_{v \in S} {^{\pm}\widetilde{K}_v(E_{p}/L_\infty)}\label{eq: fine selmer Cassels--Poitou--Tate exact sequence-}	\\
				\rightarrow  \mathscr{R}^{\pm/\pm}(E_p/ L_{\infty})\Pd \rightarrow H^2(G_{\infty}^{S}, E_{p}) \rightarrow \displaystyle\bigoplus_{w \in S(L_\infty)}  H^{2}(L_{\infty,w}, E_{p}) \rightarrow 0\nonumber.
			\end{align}
			Our plan now is to analyze this exact sequence.
			In  \cite{CoatesSujathaI}, the authors introduce the $  \mu = 0 $-conjecture for the fine Selmer group (\cf Conjecture A of [2]).
			Its original formulation asserts that for an elliptic curve $ E $ over any number filed $ F $, the Pontryagin dual of the classical fine Selmer group over the cyclotomic extension $ F_\cyc $ is a finitely generated $ \bZ_p $-module. 
			Here, we use an equivalent cohomological version of Conjecture A (\textit{cf.} \cite[Proposition 4.7]{SujFil}).
			\begin{enumerate}[{\textbf{Conjecture A}} :]\label{Conjecture A}
				\item Let $ \Gamma =  \Gal(L_\cyc/ L) $ and $ G_{\cyc}^{S} := \Gal(L^{S}/L_{\cyc})$. Then the $ \Omega(\Gamma) $-module $ H^2(G_{\cyc}^{S}, E_{p}) $ is trivial.
			\end{enumerate}
				%
			
			
			
			\subsection{Main theorem}\label{sec:Main theorem}
			We prove the main theorem (\textit{cf.} Theorem \ref{thm:TFAE}) of this article.
			This is an analogue of \cite[Proof of Theorem 4.12]{SujFil}. 
			\begin{lem}\label{lem:mu-inv-H1}
				Let $ v \in S $ and $ w | v $ be a prime in $ L_\infty $.
				\begin{enumerate}
					\item\label{lem:mu-inv-H1-1} If $ w \nmid p $, then the Pontryagin dual of $ H^{1}(L_{\infty,w}, E_{p^{\infty}}) $ has $ \mu $-invariant equal to zero.
					This implies that
					\begin{align*}
						\mathrm{corank}_{\Lambda(G)}H^{1}(L_{\infty,w}, E_{p^{\infty}}) =  \mathrm{corank}_{\Omega(G)}H^{1}(L_{\infty,w}, E_{p}).
					\end{align*}
					\item\label{lem:mu-inv-H1-2} Suppose \hyperref[Conjecture A]{Conjecture A} holds. Then, the Pontryagin dual of $ H^{1}(G^{S}_{\infty}, E_{p^{\infty}}) $ has $ \mu $-invariant equal to zero. This implies that
					\begin{align*}
						\mathrm{corank}_{\Lambda(G)}H^{1}(G_{\infty}^{S}, E_{p^{\infty}}) =  \mathrm{corank}_{\Omega(G)}H^{1}(G_{\infty}^{S}, E_{p}).
					\end{align*}
				\end{enumerate}
			\end{lem}  
			\begin{proof}
				When $ w \nmid p  $, then $ L_{\infty,w}$ is equal to  $L_{\cyc,w} $ and the statement that $ H^{1}(L_{\infty,w}, E_{p^{\infty}})\Pd $ has $ \mu $-invariant equal to zero is proven in \cite[Lemma 4.9, part i]{SujFil}. 
				By Proposition \ref{prop:psi-infty}, there is an isomorphism
				$ (H^{1}(G_{\infty}^{S}, E_{p^{\infty}}))_p \cong H^{1}(G_{\infty}^{S}, E_{p}) $
				and equation \eqref{ineq:coranks} implies 
				\begin{align*}
					\mathrm{corank}_{\Lambda(G)}H^{1}(L_{\infty,w}, E_{p^{\infty}}) =  \mathrm{corank}_{\Omega(G)}H^{1}(L_{\infty,w}, E_{p}).
				\end{align*}
				For part \hyperref[lem:mu-inv-H1-2]{(2)}, Lemma 5.6 in \cite{Lei-Sujatha} tells us 
					$ H^{1}(G_{\infty}^{S}, E_{p^{\infty}})^{H} \cong H^{1}(G_{\cyc}^{S}, E_{p^{\infty}}) $
				where $ H=\Gal(L_\infty,L_{\cyc})\cong \bZ_p $. Then $ \mu_\Gamma(H^{1}(G_{\cyc}^{S}, E_{p^{\infty}})\Pd)=0 $ by \cite[Lemma 4.9, part ii]{SujFil}.
				Let $ M=H^{1}(G_{\infty}^{S}, E_{p^{\infty}})\Pd $. 
				Note that $  \mu_\Gamma(M_H)$ vanishes 
				where $ M_H $ denotes
				the $ H $-coinvariant of $ \Lambda(G) $-module $ M $. 
				Then, a version of the topological Nakayama's lemma (\textit{cf.} \cite[Theorem 2, page 5]{Balister1997}) implies that $  \mu_G(M)=0$. 
				The equality about the coranks is implied by Proposition \ref{prop:psi-infty} and equation \eqref{ineq:coranks}. 
			\end{proof}
			Thus, given \hyperref[Conjecture A]{Conjecture A}, the last two modules in the exact sequence \eqref{eq: fine selmer Cassels--Poitou--Tate exact sequence-} vanish. 
			Therefore, to show equality \eqref{eq:cR andmathscrR have the same corank}, it is enough to show that 
			\begin{align*}
				\mathrm{corank}_{\Omega(G)}(H^1(G_{\infty}^{S}, E_{p})) = \mathrm{corank}_{\Omega(G)}(\displaystyle\bigoplus_{v \in S} {^{\pm}\widetilde{K}_v(E_{p}/L_\infty)}).
			\end{align*}
			\begin{prop}\label{prop:cR andmathscrR have the same corank}
				Let $ E/L $ be an elliptic curve satisfying \hyperref[Hyp1]{Hyp 1} and \hyperref[Hyp2]{Hyp 2$ ^{\pm/\pm} $}. 
				Moreover, suppose  \hyperref[Conjecture A]{Conjecture A} holds.
				Then
				\begin{align}\label{eq:cR andmathscrR have the same corank}
					\mathrm{corank}_{\Omega(G)}(\cR^{\pm/\pm} (E_p/L_\infty)) = \mathrm{corank}_{\Omega(G)}(\mathscr{R}^{\pm/\pm}(E_p/ L_{\infty})\Pd).
				\end{align}
			\end{prop}
			\begin{proof}
				By Theorem 3.2 and 4.1 of \cite{OchiVenjakov}, 
					$ \mathrm{corank}_{\Lambda(G)}(H^1(G_{\infty}^{S}, E_{p^{\infty}}))=[L:\bQ]=2. $
				Putting this together with Lemma \ref{lem:mu-inv-H1}-\hyperref[lem:mu-inv-H1-2]{(2)} we get
				\begin{align}\label{eq:coranks-the-same-1}
					\mathrm{corank}_{\Lambda(G)}(H^1(G_{\infty}^{S}, E_{p^{\infty}}))=\mathrm{corank}_{\Omega(G)}(H^1(G_{\infty}^{S}, E_{p}))=2.
				\end{align}
				For a prime  $ w \nmid p $ in $ L_\infty $ Proposition 2 in \cite{Greenberg1989} implies that
					$ \Lambda(G) $-corank of $ H^{1}(L_{\infty,w},E_{p^{\infty}}) $ is zero.
				This means that for $ v \nmid p $, the $ \Lambda(G) $-corank of $ J_v^{\pm}(E_{p^{\infty}}/L_\infty) $ is zero.
				Hence, we have
				\begin{align}\label{eq:coranks-the-same-2}
					\mathrm{corank}_{\Lambda(G)}(\displaystyle\bigoplus_{v \in S} J_v^{\pm}(E_{p^{\infty}}/L_\infty))
					=\mathrm{corank}_{\Lambda(G)}(\displaystyle\bigoplus_{v \in S_p} J_v^{\pm}(E_{p^{\infty}}/L_\infty)).
				\end{align}
				Furthermore, by Lemma \ref{lem:mu-inv-H1}-\hyperref[lem:mu-inv-H1-2]{(2)} we know
				\begin{align*}
					\mathrm{corank}_{\Lambda(G)}( H^{1}(L_{\infty,w},E_{p^{\infty}}))&=
					\mathrm{corank}_{\Omega(G)}H^{1}(L_{\infty,w}, E_{p})\\
					=\mathrm{corank}_{\Omega(G)}( ^{\pm}\widetilde{K}_v(E_{p}/L_\infty))&=0\\
					\implies						
					\mathrm{corank}_{\Omega(G)}(\displaystyle\bigoplus_{v \in S}   {}^{\pm}\widetilde{K}_v(E_{p}/L_\infty))
					&=\mathrm{corank}_{\Omega(G)}(\displaystyle\bigoplus_{v \in S_p}    {}^{\pm}\widetilde{K}_v(E_{p}/L_\infty))
				\end{align*}
				For $ v | p $, recall that the map $ \varphi_{v}^{\pm} $ (\textit{cf.} \eqref{eq:varphi{v}{pm}}) gives an isomorphism 
					between $ (J_v^{\pm}(E_{p^{\infty}}/L_\infty))_p  $ and $ {^{\pm}\widetilde{K}_v(E_{p}/L_\infty)} $.
				Therefore,
					$ \Omega(G) $-corank of $(J_v^{\pm}(E_{p^{\infty}}/L_\infty))_p
					$ 
					is equal to
					$
					\mathrm{corank}_{\Omega(G)}( {^{\pm}\widetilde{K}_v(E_{p}/L_\infty)} ). $
				Moreover, for $ v | p $ the local condition $ J_v^{\pm}(E_{p^{\infty}}/L_\infty) $ is a free $ \Lambda(G) $-module (\textit{cf.} \cite[Corollary 3.9]{Lei-Lim}) and
				\[
				\mathrm{corank}_{\Lambda(G)}(\displaystyle\bigoplus_{v \in S_p} J_v^{\pm}(E_{p^{\infty}}/L_\infty))=\mathrm{corank}_{\Lambda(G)}(H^1(G_{\infty}^{S}, E_{p^{\infty}})=2.
				\]
				This means that $ J_v^{\pm}(E_{p^{\infty}}/L_\infty) $ has $ \mu $-invariant zero. 
				By equation \eqref{ineq:coranks},
				\begin{align*}
					\mathrm{corank}_{\Lambda(G)}(J_v^{\pm}(E_{p^{\infty}}/L_\infty))
					&=
					\mathrm{corank}_{\Omega(G)}((J_v^{\pm}(E_{p^{\infty}}/L_\infty))_p)\\
					&=
					\mathrm{corank}_{\Omega(G)}( {^{\pm}\widetilde{K}_v(E_{p}/L_\infty)} ).
				\end{align*}
				This, together with equation \eqref{eq:coranks-the-same-2} gives us that
				\begin{align*}
					\mathrm{corank}_{\Omega(G)}(\displaystyle\bigoplus_{v \in S}   {}^{\pm}\widetilde{K}_v(E_{p}/L_\infty))
					=\mathrm{corank}_{\Omega(G)}(\displaystyle\bigoplus_{v \in S_p}    {}^{\pm}\widetilde{K}_v(E_{p}/L_\infty))=2.
				\end{align*}
				Comparing the above equation and equation \eqref{eq:coranks-the-same-1} yields
				\begin{align}\label{coranks-to-show-residual is surjective}
					\mathrm{corank}_{\Omega(G)}(H^1(G_{\infty}^{S}, E_{p})) = \mathrm{corank}_{\Omega(G)}(\displaystyle\bigoplus_{v \in S} {^{\pm}\widetilde{K}_v(E_{p}/L_\infty)})=2.
				\end{align}
				The exact sequence \eqref{eq: fine selmer Cassels--Poitou--Tate exact sequence-} implies that 
				\begin{align*}
					\mathrm{corank}_{\Omega(G)}(\cR^{\pm/\pm} (E_p/L_\infty)) = \mathrm{corank}_{\Omega(G)}(\mathscr{R}^{\pm/\pm}(E_p/ L_{\infty})\Pd).
				\end{align*}
			\end{proof}
			We are now ready to give our main theorem which is the analogue of \cite[Theorem 4.12]{SujFil} in our setting.
			It describes a criterion for the vanishing of the signed $ \mu $-invariants based completely on the structure of the fine singed residual Selmer groups as Iwasawa modules. 
			\begin{thm}\label{thm:TFAE}
				Let $ E/L $ be an elliptic curve satisfying \hyperref[Hyp1]{Hyp 1} and \hyperref[Hyp2]{Hyp 2$ ^{\pm/\pm} $}. 
				Furthermore, suppose  \hyperref[Conjecture A]{Conjecture A} holds.
				Then the following statements are equivalent:
				\begin{enumerate}
					\item\label{thm:TFAE-1} $\mathcal{Y}^{\pm/\pm}(E_{p}/L_\infty) = \cR^{\pm/\pm}(E_{p}/L_\infty)\Pd $ is $ \Omega(G) $-torsion.
					\item\label{thm:TFAE-2} 
					The $ \mu $-invariant $ \mu_G^{\pm/\pm}(E_{p^{\infty}}/L_\infty):=\mu_G(\mathfrak{X}^{\pm/\pm}(E_{p^{\infty}}/L_\infty)) $ vanishes.
					\item\label{thm:TFAE-3}
					The map  $ \xi_p^{\pm/\pm} $, described in diagram \eqref{eq: fine selmer Cassels--Poitou--Tate exact sequence-}, is surjective.
				\end{enumerate}
			\end{thm}
			\begin{proof}
				To see \eqref{thm:TFAE-1} and \eqref{thm:TFAE-2} are equivalent, 
				recall from Proposition \ref{prop:fine and p-primary of selmer have the same corank} 
				\begin{small}
					\begin{align*}
						\mathrm{rank}_{\Omega(G)}(\mathcal{Y}^{\pm/\pm}(E_{p}/L_\infty))= \mathrm{rank}_{\Omega(G)}(\mathfrak{X}^{\pm/\pm}(E_{p^{\infty}}/L_\infty) / p \mathfrak{X}^{\pm/\pm}(E_{p^{\infty}}/L_\infty));\\
						0=\mathrm{rank}_{\Lambda(G)}(\mathfrak{X}^{\pm/\pm}(E_{p^{\infty}}/L_\infty)) \leq 	\mathrm{rank}_{\Omega(G)}(\mathfrak{X}^{\pm/\pm}(E_{p^{\infty}}/L_\infty) / p \mathfrak{X}^{\pm/\pm}(E_{p^{\infty}}/L_\infty)) 
					\end{align*}
				\end{small}
				with equality exactly when  $ \mu_G^{\pm/\pm}(E_{p^{\infty}}/L_\infty) $ vanishes. 
				Therefore, given \eqref{thm:TFAE-1} we have
				\begin{small}
					\begin{align*}
						\mathrm{rank}_{\Lambda(G)}(\mathfrak{X}^{\pm/\pm}(E_{p^{\infty}}/L_\infty)) = 	\mathrm{rank}_{\Omega(G)}(\mathfrak{X}^{\pm/\pm}(E_{p^{\infty}}/L_\infty) / p \mathfrak{X}^{\pm/\pm}(E_{p^{\infty}}/L_\infty)) = 0
					\end{align*}
				\end{small}
				which implies $ \mu_G^{\pm/\pm}(E_{p^{\infty}}/L_\infty) $ is zero. 
				On the other hand, if \eqref{thm:TFAE-2} holds, then
				\begin{align*}
					0=\mathrm{rank}_{\Lambda(G)}(\mathfrak{X}^{\pm/\pm}(E_{p^{\infty}}/L_\infty)) =\mathrm{rank}_{\Omega(G)}(\mathcal{Y}^{\pm/\pm}(E_{p}/L_\infty))
				\end{align*}
				and so  \eqref{thm:TFAE-1} $ \iff $ \eqref{thm:TFAE-2}.
				Now let us show \eqref{thm:TFAE-1} $ \implies$ \eqref{thm:TFAE-3}.
				The discrete $ \Omega(G) $-module $ H^2(G_{\infty}^{S}, E_{p}) $ vanishes by Proposition 4.7 \cite{SujFil}. 
				By Proposition \ref{prop:cR andmathscrR have the same corank}
				\begin{align*}
					0=\mathrm{corank}_{\Omega(G)}(\cR^{\pm/\pm} (E_p/L_\infty)) = \mathrm{corank}_{\Omega(G)}(\mathscr{R}^{\pm/\pm}(E_p/ L_{\infty})\Pd).
				\end{align*}
				This means that $ \mathscr{R}^{\pm/\pm}(E_p/ L_{\infty})\Pd $ is a cotorsion $ \Omega(G) $-module. 
				However, being a  $\Omega(G)$-submodule of a cotorsion-free module $ 	H^{1}_{\mathrm{Iw}}(L,E_p) $, the module $\mathscr{R}^{\pm/\pm}(E_p/ L_{\infty})  $ is also torsion-free (\cf equation \eqref{eq:Iwasawacohomology}). 
				As a result, $ \mathscr{R}^{\pm/\pm}(E_p/ L_{\infty})\Pd $ vanishes.
				Therefore, the map  $ \xi_p^{\pm/\pm} $ (\textit{cf.} the exact sequence \eqref{eq: fine selmer Cassels--Poitou--Tate exact sequence-}) is surjective.
				Finally to see \eqref{thm:TFAE-3} $ \implies$ \eqref{thm:TFAE-1}, suppose  the map $ \xi_p^{\pm/\pm} $ is surjective and hence we have
				\begin{equation}\label{eq:short-exact-fine-selmer-infty}
					0 \rightarrow \cR^{\pm/\pm} (E_p/L_\infty) \rightarrow H^1(G_{\infty}^{S}, E_{p}) \xrightarrow{\xi_p^{\pm/\pm}} \displaystyle\bigoplus_{v \in S} {^{\pm}\widetilde{K}_v(E_{p}/L_\infty)} \rightarrow 0.
				\end{equation}
				Taking  Pontryagin duals of the above sequence gives
				\begin{equation*}
					0 \rightarrow 
					\displaystyle\bigoplus_{v \in S} {^{\pm}\widetilde{K}_v(E_{p}/L_\infty)}\Pd
					\xrightarrow{{\xi_p^{\pm/\pm}}\Pd} 
					H^1(G_{\infty}^{S}, E_{p})\Pd  \rightarrow
					\mathcal{Y}^{\pm/\pm}(E_{p}/L_\infty) \rightarrow
					0.
				\end{equation*}
				In the proof of Proposition \ref{prop:cR andmathscrR have the same corank}, we proved that (\textit{cf.} equation \eqref{coranks-to-show-residual is surjective})
				the first two terms in the above short exact sequence have the same $ \Omega(G) $-rank which is equal to $ 2 $. 
				So, $ \mathcal{Y}^{\pm/\pm}(E_{p}/L_\infty) $ is a torsion $ \Omega(G)  $-module.
			\end{proof}
			We record the following important corollary of Theorem \ref{thm:TFAE}.
			\begin{cor}\label{cor:main}
				Let  $E_1/L $ and $ E_2/L $ be two elliptic curves that satisfy \hyperref[Hyp1]{Hyp 1} and \hyperref[Hyp2]{Hyp 2$ ^{\pm/\pm} $}.
				Suppose they have isomorphic residual Galois representations and \hyperref[Conjecture A]{Conjecture A} is satisfied for either $E_1/L $ or $ E_2/L $ (and hence both).  
				Then,  
					$ 	\mu_G^{\pm/\pm}((E_1)_{p^{\infty}}/L_\infty) =0 $ if and only if $ \mu_G^{\pm/\pm}((E_2)_{p^{\infty}}/L_\infty) =0. $
			\end{cor} 
			\begin{proof}
				Without loss of generality assume $ \mu_G^{\pm/\pm}((E_1)_{p^{\infty}}/L_\infty) $  vanishes.
				{ 
					Note that \hyperref[Conjecture A]{Conjecture A} for an elliptic curve $ E/L $ only depends on the isomorphism class of residual Galois representation $ E_p $.
				}
				By Theorem \ref{thm:TFAE}, the module 
				$ \mathcal{Y}^{\pm/\pm}((E_1)_{p}/L_\infty)$ is $ \Omega(G) $-torsion. 
				Proposition \ref{prop:fine selmer only depends on residual reps} gives
				\[
				\cR^{\pm/\pm} ((E_1)_p/L_\infty) \cong \cR^{\pm/\pm} ((E_2)_p/L_\infty).
				\]
				Thus, the module $ \mathcal{Y}^{\pm/\pm}((E_2)_{p}/L_\infty)$ is also $ \Omega(G) $-torsion. 
				Again by  Theorem \ref{thm:TFAE}, this implies that $ \mu_G^{\pm/\pm}((E_2)_{p^{\infty}}/L_\infty) $ is equal to zero.
			\end{proof}
			
			\subsection{Comparison with the cyclotomic level}
			Let $ H$ denote the Galois group $\Gal(L_\infty/ L_{\cyc}) $ and consider the following commutative diagram
			\begin{small}
				\begin{equation}\label{eq:commutative-diagram-R{pm/pm}}
					\begin{tikzcd}
						0 \arrow[r] & \cR^{\pm/\pm} (E_p/L_\cyc) \arrow[r] \arrow[d, "\alpha_p^{\pm/\pm}"']          & H^1(G_{\cyc}^{S}, E_{p}) \arrow[r, "\xi_{p,\cyc}^{\pm/\pm}"] \arrow[d, "\beta_p"']    & \displaystyle\bigoplus_{v \in S} {^{\pm}\widetilde{K}_v(E_{p}/L_\cyc)} \arrow[d, "\widetilde{\gamma}^{\pm/\pm}_p:={\ds\bigoplus_{v \in S}\widetilde{\gamma}_{v}^{\pm}}"] \\
						0 \arrow[r] & \cR^{\pm/\pm} (E_p/L_\infty)^{H} \arrow[r] & H^1(G_{\infty}^{S}, E_{p})^{H} \arrow[r, "\xi_{p}^{\pm/\pm,H}"'] & \displaystyle\bigoplus_{v \in S} {^{\pm}\widetilde{K}_v(E_{p}/L_\infty)}^{H}.              	
					\end{tikzcd}
				\end{equation} 
			\end{small}

				\begin{prop}\label{prop:all vertical maps in commutative-diagram-R{pm/pm} are iso}
					All the vertical maps in the diagram \eqref{eq:commutative-diagram-R{pm/pm}} are isomorphisms. 
					In particular, 
						$ \cR^{\pm/\pm} (E_p/L_\infty)^{H}  \cong \cR^{\pm/\pm} (E_p/L_\cyc) $.
				\end{prop}
				\begin{proof}
					The map $ \beta_p $  is an isomorphism by \cite[Lemma 5.6]{Lei-Sujatha}. 
					Let us show that the map $ \widetilde{\gamma}^{\pm/\pm}_p:={\bigoplus_{v \in S}\widetilde{\gamma}_{v}^{\pm}} $ is an isomorphism.
					{For any $ v \in S \backslash S_p $, the map $ \widetilde{\gamma}_{v}^{\pm}  $ is  the identity map (\textit{cf.} proof of Lemma 5.10 in \cite{Lei-Sujatha}).}
					Let us assume $ v  = \fp $
					and let $ A:= J^{\pm}_\fp (E_{p^{\infty}}/L_{\infty}) $.
					Note that, by Corollary \ref{cor:psi-infty-fq}, $ A_p \cong {}^{\pm}\widetilde{K}_\fp(E_{p}/L_\infty) $ and by Lemma 5.10 of \cite{Lei-Sujatha}, $ A^{H} \cong J^{\pm}_\fp (E_{p^{\infty}}/L_{\cyc}) $.
					Furthermore, by \cite[Proposition 4.1-d]{SujFil}, there is the following isomorphism 
					\begin{align*}
						\psi_{\fq,\cyc}^{\pm} :	{}^{\pm}\widetilde{K}_\fp(E_{p}/L_\cyc) 
						\xrightarrow{\cong} 
						\left(J^{\pm}_\fp (E_{p^{\infty}}/L_{\cyc})\right)_p
						.
					\end{align*}
					Using the isomorphism $ (A^{H})_p \cong (A_p)^{H} $ along with the above map, we see
					\begin{small}
						\begin{equation*} 
							\widetilde{\gamma}_{\fp}^{\pm} : {}^{\pm}\widetilde{K}_\fp(E_{p}/L_\cyc) \xRightarrow[\psi_{\fq,\cyc}^{\pm}]{\cong} \left(J^{\pm}_\fp (E_{p^{\infty}}/L_{\cyc})\right)_p \xRightarrow{\cong} (A^{H})_p \cong (A_p)^{H} \xRightarrow{\cong} {}^{\pm}\widetilde{K}_\fp(E_{p}/L_\infty)^{H}
						\end{equation*}
					\end{small}
					is an isomorphism.
					The case where $   v  = \bfp$ is similar.
					Finally, the snake lemma implies that  $ \alpha_p^{\pm/\pm} $ is also an isomorphism.
				\end{proof}
				We record the analogue of Theorem \ref{thm:TFAE} in the cyclotomic setting.
				\begin{thm}(\cite[Theorem 4.12]{SujFil})\label{thm:TFAE-cyc}
					Let $ E/L $ be an elliptic curve satisfying \hyperref[Hyp1]{Hyp 1} and \hyperref[Hyp2(cyc)]{Hyp 2$ ^{\pm/\pm} $(cyc)}. 
					Then the following are equivalent:
					\begin{enumerate}
						\item\label{thm:TFAE-1-cyc} $\mathcal{Y}^{\pm/\pm}(E_{p}/L_\cyc) = \cR^{\pm/\pm}(E_{p}/L_\cyc)\Pd $ is $ \Omega(\Gamma) $-torsion.
						\item\label{thm:TFAE-2-cyc} 
						The signed cyclic $ \mu $-invariant $ \mu_\Gamma^{\pm/\pm}(E_{p^{\infty}}/L_\cyc):=\mu_\Gamma(\mathfrak{X}^{\pm/\pm}(E_{p^{\infty}}/L_\cyc)) $ vanishes.
						\item\label{thm:TFAE-3-cyc} The map  $ \xi_{p,\cyc}^{\pm/\pm} $ in diagram \eqref{eq:commutative-diagram-R{pm/pm}} is surjective and \hyperref[Conjecture A]{Conjecture A} holds.
					\end{enumerate}
				\end{thm}
				\begin{cor}\label{cor:commutative-diagram-R{pm/pm}-iso}
					Suppose $ E/L $ satisfies \hyperref[Hyp1]{Hyp 1} and \hyperref[Hyp2(cyc)]{Hyp 2$ ^{\pm/\pm} $(cyc)}. 
					Assume one of the equivalent statements in Theorem \ref{thm:TFAE-cyc} is satisfied. Then, the top and bottom exact sequences in diagram \eqref{eq:commutative-diagram-R{pm/pm}} are short exact. 
				\end{cor}
				\begin{proof}
					By Theorem \ref{thm:TFAE-cyc}, the map $\xi_{p,\cyc}^{\pm/\pm}  $ is surjective and hence the top row in diagram \eqref{eq:commutative-diagram-R{pm/pm}} is short exact. 
					Since the map $ \widetilde{\gamma}^{\pm/\pm}_p \circ \xi_{p,\cyc}^{\pm/\pm}$ is surjective, the map $ \xi_{p}^{\pm/\pm,H} $ is also surjective.
				\end{proof}
				We now combine Theorem \ref{thm:TFAE-cyc} and Proposition \ref{prop:all vertical maps in commutative-diagram-R{pm/pm} are iso}.
				\begin{thm}\label{thm:TFAE-cyc implies TFAE}	
					Suppose $ E/L $ satisfies \hyperref[Hyp1]{Hyp 1} and \hyperref[Hyp2(cyc)]{Hyp 2$ ^{\pm/\pm} $(cyc)}.
					Furthermore, assume one of the equivalent statements in Theorem \ref{thm:TFAE-cyc} is satisfied.
					Then, the equivalent statements in Theorem \ref{thm:TFAE} are satisfied. 
				\end{thm}
				\begin{proof}
					\hyperref[Conjecture A]{Conjecture A} is satisfied by part \ref{thm:TFAE-3-cyc} of Theorem \ref{thm:TFAE-cyc}.
					Proposition \ref{pro:Hyp 2 (cyc) implies Hyp 2} implies that \hyperref[Hyp2]{Hyp 2$ ^{\pm/\pm} $} holds.
					Using Proposition \ref{prop:all vertical maps in commutative-diagram-R{pm/pm} are iso}
					\[
					\mathcal{Y}^{\pm/\pm}(E_{p}/L_\cyc) =\cR^{\pm/\pm}(E_{p}/L_\cyc)\Pd\cong (\cR^{\pm/\pm}(E_{p}/L_\infty)^{H})\Pd \cong \mathcal{Y}^{\pm/\pm}(E_{p}/L_\infty)_H.
					\]
					Suppose $ \mathcal{Y}^{\pm/\pm}(E_{p}/L_\cyc) $ which is isomorphic to $  \mathcal{Y}^{\pm/\pm}(E_{p}/L_\infty)_H $ is torsion as a $ \Omega(\Gamma) $-module. 
					Then, we can apply \cite[Lemma 2.6]{Hachimori2010} to get that the module $ \mathcal{Y}^{\pm/\pm}(E_{p}/L_\infty) $ is torsion as a $ \Omega(G) $-module. 	
				\end{proof}
				\begin{cor}\label{cor:H1(R)=0}
					With the same assumption as Theorem \ref{thm:TFAE-cyc implies TFAE}, 
					\begin{align*}
						H^{1}(H,\cR^{\pm/\pm} (E_p/L_\infty))=0.
					\end{align*}
				\end{cor}
				\begin{proof}
					By Theorem \ref{thm:TFAE-cyc implies TFAE}, we have the short exact sequence \eqref{eq:short-exact-fine-selmer-infty}.
					Taking the long exact $ H $-Galois cohomology of this sequence yields 
					$  \Coker(\xi_{p}^{\pm/\pm,H})$ is equal to $ H^{1}(H,\cR^{\pm/\pm} (E_p/L_\infty)) .$
					By Corollary \ref{cor:commutative-diagram-R{pm/pm}-iso}, the map $ \xi_{p}^{\pm/\pm,H} $ is surjective which proves the claim.
				\end{proof}
				\begin{cor}\label{cor:TFAE-cyc implied by TFAE}
					Suppose $ E/L $ is an elliptic curve satisfying \hyperref[Hyp1]{Hyp 1} and \hyperref[Hyp2(cyc)]{Hyp 2$ ^{\pm/\pm} $(cyc)}. 
					Furthermore, assume the $ \Omega(\Gamma) $-module $ H^{1}(H,\cR^{\pm/\pm} (E_p/L_\infty)) $ vanishes. 
					Then, the converse of Theorem \ref{thm:TFAE-cyc implies TFAE} holds too.
					That is, all the statements in Theorem \ref{thm:TFAE-cyc} and Theorem \ref{thm:TFAE} are equivalent.
				\end{cor}
				\begin{proof}
					Suppose the map  $ \xi_p^{\pm/\pm} $ is surjective and \hyperref[Conjecture A]{Conjecture A} holds. 
					Since $ H^{1}(H,\cR^{\pm/\pm} (E_p/L_\infty)) $ vanishes, a similar argument as in  Corollary \ref{cor:H1(R)=0} shows that the map $ \xi_{p}^{\pm/\pm,H} $ is surjective.
					Using diagram \eqref{eq:commutative-diagram-R{pm/pm}}, the map $\xi_{p}^{\pm/\pm,H} \circ \beta_p  $ is surjective, and hence that the map $ \xi_{p,\cyc}^{\pm/\pm} $ is also surjective. 
				\end{proof}
					Theorem \ref{thm:TFAE-cyc implies TFAE} implies that if $ E/L $ satisfies  \hyperref[Hyp2(cyc)]{Hyp 2$ ^{\pm/\pm} $(cyc)}, then 
					\begin{align}\label{eq:mu-cyc-implies-mu-infty}
						\mu_\Gamma^{\pm/\pm}(E_{p^{\infty}}/L_\cyc) = 0 \implies  \mu_G^{\pm/\pm}(E_{p^{\infty}}/L_\infty)=0.
					\end{align}
					Moreover, if $ H^{1}(H,\cR^{\pm/\pm} (E_p/L_\infty)) $ vanishes and  \hyperref[Conjecture A]{Conjecture A} holds then the two sides of the statement \eqref{eq:mu-cyc-implies-mu-infty} become equivalent by Corollary \ref{cor:TFAE-cyc implied by TFAE}.
				\section{Pseudo-null submodules}\label{sec:Pseudo-null submodules}
				Let $ E/L $ be an elliptic curve satisfying \hyperref[Hyp1]{Hyp 1}.
				We show that under some mild assumptions, 
				the $ \Lambda(G) $-modules the $ \Lambda(G) $-modules $ \mathfrak{X}^{\pm/\pm}(E_{p^{\infty}}/L_\infty) $ and $ \mathfrak{X}(E_{p^{\infty}}/L_\infty) $ have no non-trivial pseudo-null submodules 	(\cf Definition 5.1.4 of \cite{Neukirch} for the definition of pseudo-null submodules).

				
				{The Iwasawa algebra $\Lambda(G)$ is a Noetherian regular local commutative ring and hence a Cohen--Macaulay ring. 
					In particular, the depth of $ \Lambda(G) $ coincides with its Krull dimension (\textit{cf.} Chapter 17 of \cite{Eisenbud1995}.). 
					By the Auslander--Buchsbaum--Serre theorem, the global dimension of $ \Lambda(G) $ is finite and it coincides with its Krull dimension.}
			Let $ H$ be $\Gal(L_\infty,L_{\cyc}) $ and consider the following diagram of $ \Lambda(\Gamma) $-modules:
			\begin{small}
				\begin{equation}\label{eq:fundumetal diagram for Selmer}
					\begin{tikzcd}
						0 \arrow[r] &  \Sel^{\pm/\pm} (E_{p^{\infty}}/L_\cyc) \arrow[r] \arrow[d, "\alpha^{\pm/\pm}"', "\cong"]          & H^1(G_{\cyc}^{S}, E_{p^{\infty}}) \arrow[r, "\xi_\cyc^{\pm/\pm}" ] \arrow[d, "\beta"', "\cong"]    & \displaystyle\bigoplus_{v \in S} J_v^{\pm}(E_{p^{\infty}}/L_\cyc) \arrow[d, "{\gamma}^{\pm/\pm}:={\ds\bigoplus_{v \in S}{\gamma}_{v}^{\pm}}", "\cong"'] \\
						0 \arrow[r] & \Sel^{\pm/\pm} (E_{p^{\infty}}/L_\infty)^{H} \arrow[r] & H^1(G_{\infty}^{S}, E_{p^{\infty}})^{H} \arrow[r, "\xi^{\pm/\pm,H}"] & \displaystyle\bigoplus_{v \in S} \left(J_v^{\pm}(E_{p^{\infty}}/L_\infty)\right)^{H}.           	
					\end{tikzcd}
				\end{equation}
			\end{small}
			The maps $\beta $ and ${\gamma}^{\pm/\pm}  $ are isomorphisms by Lemma 5.6, 5.8, and 5.10 of \cite{Lei-Sujatha}.
			By snake lemma, the map $ \alpha^{\pm/\pm} $ is also an isomorphism and hence, all the vertical maps in the diagram \eqref{eq:fundumetal diagram for Selmer} are isomorphisms. 
			\begin{prop}\label{pro:Hyp 2 (cyc) implies Hyp 2}
				Suppose $ E/L $ is an elliptic curve over quadratic number field $ L $ satisfying \hyperref[Hyp1]{Hyp 1}.
				If \hyperref[Hyp2(cyc)]{Hyp 2$ ^{\pm/\pm} $(cyc)} holds then
				\begin{align}\label{eq:mathfrak{X}{pm/pm}(infty)-H iso mathfrak{X}{pm/pm}(cyc)}
					\mathfrak{X}^{\pm/\pm}(E_{p^{\infty}}/L_\cyc) \cong \mathfrak{X}^{\pm/\pm} (E_{p^{\infty}}/L_\infty)_H,
				\end{align}
				where $ \mathfrak{X}^{\pm/\pm} (E_{p^{\infty}}/L_\infty)_H  $ denotes the of $ H $-coinvariants of $ \mathfrak{X}^{\pm/\pm} (E_{p^{\infty}}/L_\infty)$.
				Moreover, \hyperref[Hyp2]{Hyp 2$ ^{\pm/\pm} $} holds and  the group	$ H^{1}(H,\Sel^{\pm/\pm} (E_{p^{\infty}}/L_\infty)) $ vanishes.
			\end{prop}
			\begin{proof}
				By Proposition \ref{pro:Hyp 2 equivalence}, the top row of the diagram \eqref{eq:fundumetal diagram for Selmer} becomes a short exact sequence. 
				A diagram chase gives that, since the map ${\gamma}^{\pm/\pm}  \circ \xi_\cyc^{\pm/\pm}  $ is surjective, the map $ \xi^{\pm/\pm,H} $ is also surjective and hence, the diagram \eqref{eq:fundumetal diagram for Selmer} becomes an isomorphism of two short exact sequences.
				Taking the Pontryagin dual of the isomorphism $ \alpha^{\pm/\pm} $ gives the first claim.

				By \hyperref[Hyp2(cyc)]{Hyp 2$ ^{\pm/\pm} $(cyc)}, $ \Lambda(\Gamma) $-module $ \mathfrak{X}^{\pm/\pm} (E_{p^{\infty}}/L_\infty)_H $ is finitely generated and torsion. 
				This implies that $ \mathfrak{X}^{\pm/\pm} (E_{p^{\infty}}/L_\infty) $ is a finitely generated torsion $ \Lambda(G) $-module (\textit{cf. \cite[Lemma 2.6]{Hachimori2010}}) which is \hyperref[Hyp2]{Hyp 2$ ^{\pm/\pm} $}. 
				By Proposition \ref{pro:Hyp 2 equivalence}, this implies that the map $\xi^{\pm/\pm}   $ in \eqref{eq:signed Selmer infty with J_v pm} is surjective.
				Taking long exact $ H $-cohomology of the short exact sequence  \eqref{eq:signed Selmer infty with J_v pm} , we get the bottom short exact sequence in the diagram \eqref{eq:fundumetal diagram for Selmer}.   
				In particular, the map $\xi^{\pm/\pm,H}$  is surjective which implies 
					$ 	\Coker(\xi^{\pm/\pm,H}) =H^{1}(H,\Sel^{\pm/\pm} (E_{p^{\infty}}/L_\infty))   = 0. $
			\end{proof}
			\begin{lem}
				Assuming \hyperref[Hyp2(cyc)]{Hyp 2$ ^{\pm/\pm} $(cyc)} holds, the cohomological group 
				\newline
				$ H^{1}(\Gamma,\Sel^{\pm/\pm} (E_{p^{\infty}}/L_\cyc)) $ vanishes if and only if $ \xi_{\cyc}^{\pm/\pm,\Gamma}$ is surjective. 
			\end{lem}
			\begin{proof}
				By Proportion \ref{pro:Hyp 2 (cyc) implies Hyp 2}, $ H^{1}(H,\Sel^{\pm/\pm} (E_{p^{\infty}}/L_\infty))  $ vanishes. 
				The inflation-restriction exact sequence implies that 
				\begin{align*}
					H^{1}(\Gamma,\Sel^{\pm/\pm} (E_{p^{\infty}}/L_\cyc)) = H^{1}(G,\Sel^{\pm/\pm} (E_{p^{\infty}}/L_\infty)).
				\end{align*}
				Let $ G^{S}:=\Gal(L^{S}/L) $ and consider the following commutative diagram
				\begin{small}
					\begin{equation}\label{eq:fundumetal diagram for finite Selmer}
						\begin{tikzcd}
							0 \arrow[r] &  \Sel (E_{p^{\infty}}/L) \arrow[r] \arrow[d, "\alpha^{\pm/\pm}_{\cyc}"']          & H^1(G^{S}, E_{p^{\infty}}) \arrow[r, "\lambda_0" ] \arrow[d, "\beta_{\cyc}"']    & \displaystyle\bigoplus_{v \in S} J_v(E_{p^{\infty}}/L) \arrow[d, "{\gamma}_{\cyc}^{\pm/\pm}:={\ds\bigoplus_{v \in S}{\gamma}_{\cyc,v}^{\pm}}"]\\
							0 \arrow[r] & \Sel^{\pm/\pm} (E_{p^{\infty}}/L_\cyc)^{\Gamma} \arrow[r] & H^1(G_{\cyc}^{S}, E_{p^{\infty}})^{\Gamma} \arrow[r, "\xi_{\cyc}^{\pm/\pm,\Gamma}"] & \displaystyle\bigoplus_{v \in S} \left(J_v^{\pm}(E_{p^{\infty}}/L_\cyc)\right)^{\Gamma}.           	
						\end{tikzcd}
					\end{equation}
				\end{small}
				The map $ \lambda_{0} $ is as in the exact sequence \eqref{def:Selmer with J_n}. 
				Note that at $  L_0 = L $, the classical $ p^{\infty} $-Selmer and the signed Selmer groups coincide.
				Recall that \hyperref[Hyp2(cyc)]{Hyp 2$ ^{\pm/\pm} $(cyc)}  implies the map $\xi^{\pm/\pm}   $ is surjective by Proposition \ref{pro:Hyp 2 equivalence}.
				Taking long exact $ \Gamma $-cohomology  of the short exact sequence \eqref{eq:short-exact-selmer-cyc} yields that 
				$ \Coker(\xi_{\cyc}^{\pm/\pm,\Gamma} ) = H^{1}(\Gamma,\Sel^{\pm/\pm} (E_{p^{\infty}}/L_\cyc)) $.
			\end{proof}
			\begin{lem}\label{rem:beta-{cyc} is iso}
				The map $ \beta_{\cyc} $ in diagram \eqref{eq:fundumetal diagram for finite Selmer} is an isomorphism. 
			\end{lem}
			\begin{proof}
				The Hochschild--Serre spectral sequence gives
				\begin{align*}
					0 \rightarrow H^{1}(\Gamma , E_{p^{\infty}}(L_{\cyc}))    \rightarrow H^1(G^{S}, E_{p^{\infty}}) \xrightarrow{\beta_{\cyc}} H^{1}(G_{\cyc}^{S},E_{p^{\infty}})^{\Gamma} \rightarrow  0 	
				\end{align*}
				The last term is zero as $\Gamma  $ has $ p $-cohomological dimension.
				Moreover, equation \eqref{eq:Ep=0} gives that $E_{p^{\infty}}(L_{\cyc})  $ is zero and so 
				$ H^{1}(\Gamma , E_{p^{\infty}}(L_{\cyc})) $ vanishes.
			\end{proof}
			\begin{rem}(\cite[Corollary 5.13]{Lei-Sujatha})\label{prop:H1(Gamma,Sel{pm/pm} (E{p{infty}}/Lcyc))=0}
				If $ \Sel(E_{p^{\infty}}/L) $ is finite, the map $ \xi_{\cyc}^{\pm/\pm,\Gamma}$ is surjective and hence
				\[
				H^{1}(\Gamma,\Sel^{\pm/\pm} (E_{p^{\infty}}/L_\cyc)) = H^{1}(G,\Sel^{\pm/\pm} (E_{p^{\infty}}/L_\infty)) = 0.
				\] 
			\end{rem}
			Let $ \gamma $ and $ h $ be topological generators of the groups $ \Gamma $ and $ H $ respectively.
			Since $ G \cong H \times \Gamma$, there is an isomorphism
			\begin{align*}
				\Lambda(G)=\bZ_p[[G]]\cong \bZ_p[[H \times \Gamma]] \xrightarrow{\cong} \bZ_{p}[[T_1,T_2]]
			\end{align*}
			where $ h-1 $ (\textit{resp.} $ \gamma-1 $)  is mapped to the indeterminate variable $ T_1 $ (\textit{resp.} $ T_2 $ ) and we extend $ \bZ_p $-linearly.
			In what follows we use $ \Lambda(H) $ (\textit{resp.} $ \Lambda(\Gamma) $) and $ \bZ_p[[T_1]] $ (\textit{resp.} $  \bZ_p[[T_2]] $) interchangeably. 
			\begin{definition}\label{def:regular sequence}
				Suppose $ R $ is a ring and let $ M $ be an $ R $-module. 
				A sequence of elements $ f_1,\cdots,f_k $ of $ R $ is called an \textbf{$ M $-regular sequence} if the following conditions hold:
				\begin{enumerate}
					\item The element $ f_i $ is a non-zero divisor on $ M/(f_1,\cdots,f_{i-1})M $ for each $ i=1,\cdots,k $;
					\item the module $ M/(f_1,\cdots,f_k)M $ is not zero.	
				\end{enumerate}
				If $ I $ is an ideal of the ring $ R $ and $ f_1,\cdots,f_k \in I $ then we call $ f_{1},…,f_{k} $ an \textbf{$ M $-regular sequence in $ I $}. 
				Moreover, if $ M=R $, we call $ f_{1},\cdots,f_{k} $ a \textbf{regular sequence in $ I $}.
			\end{definition}
			Using the above definition we can define the depth of a finitely generated module over an Iwasawa algebra $ \Lambda $.
			\begin{definition}\label{def:depth} 
				Let $ I $ be an ideal of $ \Lambda $ and suppose $ M $ is a finitely generated $ \Lambda $-module such that $ IM \neq M $. 
				Then the \textbf{$ I $-depth} of $ M $, denoted by $ \depth_I(M) $, is the maximal length of a $ M $-regular sequence in $ I $. 
				The \textbf{depth} of $ M $ as an $  \Lambda$-module, denoted by $ \depth(M) $, is defined to be
				\[
				\depth(M) := \depth_\fm(M)
				\]
				where $ \fm $ is the maximal ideal of $ \Lambda $.
			\end{definition}
			{
				\begin{thm}\label{thm:no no pseudo-null}
					Suppose $E/L  $ is an elliptic curve satisfying \hyperref[Hyp1]{Hyp 1} and $ \Sel(E_{p^{\infty}}/L) $ is finite.
					{Then the following assertions hold.}
					\begin{enumerate}
						\item\label{thm:no no pseudo-null-part1}  $ \mathfrak{X}^{\pm/\pm} (E_{p^{\infty}}/L_\infty) $ has no non-trivial pseudo-null $ \Lambda(G) $-submodule.
						\item\label{thm:no no pseudo-null-part2}  $ \mathfrak{X}(E_{p^{\infty}}/L_\infty) $ has no non-trivial pseudo-null $ \Lambda(G) $-submodule.
					\end{enumerate}
				\end{thm}
				\begin{proof}
					For part \eqref{thm:no no pseudo-null-part1}, note that  for a finitely generated $ \Lambda(G) $-module $ M $
					$$ \depth(M) \leq \T{Krulldim}(M) \leq \T{Krulldim}(\Lambda(G)) =3. $$ 
					Since the module $ \mathfrak{X}^{\pm/\pm} (E_{p^{\infty}}/L_\infty) $ is $ \Lambda(G) $-torsion then  $ \T{depth}(\mathfrak{X}^{\pm/\pm} (E_{p^{\infty}}/L_\infty)) \leq 2$. 
					If the depth of $ \mathfrak{X}^{\pm/\pm} (E_{p^{\infty}}/L_\infty)$ is two, then the Auslander--Buchsbaum formula (\textit{cf.} Chapter 19.3 of \cite{Eisenbud1995}) will imply that the projective dimension of $  \mathfrak{X}^{\pm/\pm} (E_{p^{\infty}}/L_\infty) $ is equal to one.
					By Proposition 3.10 of \cite{venjakob2002structure}, any module of projective dimension at most one has no non-trivial non-trivial pseudo-null submodule.
					By Proposition \ref{pro:Hyp 2 (cyc) implies Hyp 2}, $ H{^{1}(H,\Sel^{\pm/\pm} (E_{p^{\infty}}/L_\infty))}\Pd  $ vanishes, and
					\begin{align}\label{lem:Tor{1} is MT1}
						0=H{^{1}(H,\Sel^{\pm/\pm} (E_{p^{\infty}}/L_\infty))}\Pd \nonumber
						&\cong
						H_{1}(H,\mathfrak{X}^{\pm/\pm} (E_{p^{\infty}}/L_\infty))\\
						=\Tor_{1}^{\Lambda(H)}(\bZ_p,\mathfrak{X}^{\pm/\pm} (E_{p^{\infty}}/L_\infty)) 
						&= \mathfrak{X}^{\pm/\pm} (E_{p^{\infty}}/L_\infty)[T_1]
					\end{align}
					where $ \mathfrak{X}^{\pm/\pm} (E_{p^{\infty}}/L_\infty)[T_1] $ is the set of elements in $  \mathfrak{X}^{\pm/\pm} (E_{p^{\infty}}/L_\infty)$ that are annihilated by $ T_1 $.
					Therefore, the element $ T_1 $ is a non-zero divisor {on} $ \mathfrak{X}^{\pm/\pm} (E_{p^{\infty}}/L_\infty) $. 
					{	
						Proposition \ref{pro:Hyp 2 (cyc) implies Hyp 2} also tells us that
							$ 	\mathfrak{X}^{\pm/\pm} (E_{p^{\infty}}/L_\infty)_H \cong \mathfrak{X}^{\pm/\pm}(E_{p^{\infty}}/L_\cyc)  $
						as $ \Lambda(\Gamma) $-modules.
						Note that $ \mathfrak{X}^{\pm/\pm} (E_{p^{\infty}}/L_\infty)_H  $ is non-zero, otherwise, by Nakayama's lemma $ \mathfrak{X}^{\pm/\pm} (E_{p^{\infty}}/L_\infty) $  is trivial as  $ T_1 \in  \fm_G $.
						Arguing as above, we see
						\begin{small}
							\begin{align*}
								H^{1}(\Gamma,\Sel^{\pm/\pm} (E_{p^{\infty}}/L_\cyc))\Pd 
								\cong
								H_1(\Gamma,\mathfrak{X}^{\pm/\pm} (E_{p^{\infty}}/L_\infty)_H)
								= (\mathfrak{X}^{\pm/\pm} (E_{p^{\infty}}/L_\infty)_H)[T_2].
							\end{align*}
						\end{small}
						By Remark \ref{prop:H1(Gamma,Sel{pm/pm} (E{p{infty}}/Lcyc))=0},  the group $ H^{1}(\Gamma,\Sel^{\pm/\pm} (E_{p^{\infty}}/L_\cyc)) $ vanishes.
						Hence, $ T_2 $ is a non-zero divisor {on}  $ \mathfrak{X}^{\pm/\pm} (E_{p^{\infty}}/L_\infty)_H  $.}
					Note that using Nakayama's lemma  we can see that the $ \Lambda(G) $-module 
					$ 	\mathfrak{X}^{\pm/\pm} (E_{p^{\infty}}/L_\infty)/(T_1,T_2)\mathfrak{X}^{\pm/\pm} (E_{p^{\infty}}/L_\infty) $ is non-trivial.
					{Therefor, the set $ \{T_1,T_2\} $ is an $ \fm_{G} $-sequence on $ \mathfrak{X}^{\pm/\pm} (E_{p^{\infty}}/L_\infty) $.
						Since $ \mathfrak{X}^{\pm/\pm} (E_{p^{\infty}}/L_\infty) $ is  $ \Lambda(G) $-torsion, then  $ \T{depth}(\mathfrak{X}^{\pm/\pm} (E_{p^{\infty}}/L_\infty))$ is two.}
					
					Similar to the first part, to show part \eqref{thm:no no pseudo-null-part2}, it suffices to show that $ \mathfrak{X} (E_{p^{\infty}}/L_\infty) $ has projective dimension at most one.
					Consider the short exact sequence:
					\begin{align*}
						0 \longrightarrow 
						\Sel^{\pm/\pm} (E_{p^{\infty}}/L_\infty) \longrightarrow \Sel (E_{p^{\infty}}/L_\infty) \longrightarrow \bigoplus_{v \in \{\fp,\bfp\}} J_v^{\pm}(E_{p^{\infty}}/L_\infty)
						\longrightarrow 0.
					\end{align*}
					The last map is exact by Proposition \ref{pro:Hyp 2 equivalence}.
					Let $ 	U:= \bigoplus_{v \in \{\fp,\bfp\}} J_v^{\pm}(E_{p^{\infty}}/L_\infty)\Pd $, then
					taking the Pontryagin dual of the above sequence yields
					\begin{align}\label{eq:signed Selmer and Selmer infty with J_v pm}
						0 \longrightarrow 
						U \longrightarrow \mathfrak{X} (E_{p^{\infty}}/L_\infty) \longrightarrow  \mathfrak{X}^{\pm/\pm} (E_{p^{\infty}}/L_\infty)
						\longrightarrow 0.
					\end{align}
					By Corollary 3.9 of \cite{Lei-Lim}, the $ \Lambda(G) $-module $ U $ is free of rank $ 2 $ and hence it has  projective dimension zero. 
					By part \eqref{thm:no no pseudo-null-part1}, the depth of $ \mathfrak{X}^{\pm/\pm} (E_{p^{\infty}}/L_\infty)  $ is equal to two, and so by the Auslander--Buchsbaum--Serre formula it has projective dimension one.
					Now, applying Lemma  10.109.9 of \cite{stacks-project} to the short exact sequence \eqref{eq:signed Selmer and Selmer infty with J_v pm} shows that the projective dimension of $ \mathfrak{X} (E_{p^{\infty}}/L_\infty) $ is at most one.

				\end{proof}
				Let us end this section by relating the signed $ \mu $-invariant $ \mu_G^{\pm/\pm}(E_{p^{\infty}}/L_\infty) $ to
				the $ \mu $-invariant of the torsion $ \Lambda(G) $-submodule of the Pontryagin dual of the Selmer group.
				Suppose $ M $ is a finitely generated $ \Lambda(G) $-module and  let $ T_{\Lambda(G) }(M) $ denote the $ \Lambda(G)  $-torsion submodule of $ M $.
				Note that
				\begin{align*}
					T_{\Lambda(G)}(M) = \ker\left( M \to M \ot_{\Lambda(G)} K \right),
				\end{align*}
				where $ K = \T{Frac}(\Lambda(G)) $ is the fraction field of $ \Lambda(G) $.
				With abuse of terminology, by the $ \mu$-invariant of $ M $, we mean the $ \mu $-invariant of the $ \Lambda(G)  $-torsion submodule of $ M $.
				\begin{prop}\label{prop:signed_mu_implies_mu-vanishing}
					Suppose $E/L  $ is an elliptic curve satisfying \hyperref[Hyp1]{Hyp 1} and \hyperref[Hyp2]{Hyp 2$ ^{\pm/\pm} $}.
					Then, the $ \mu $-invariant of 
					$\mathfrak{X}(E_{p^{\infty}}/L_\infty)  $ is bounded by the signed $ \mu $-invariant $ \mu_G^{\pm/\pm}(E_{p^{\infty}}/L_\infty) $.
				\end{prop}
				\begin{proof}
					Noting that $ K $ is a flat $ \Lambda(G) $-module
					Thus, when we apply the functor $ - \ot_{\Lambda(G)} K $
					to the short exact sequence \eqref{eq:signed Selmer and Selmer infty with J_v pm}, we obtain the following diagram:
					{\small
						\[	\begin{tikzcd}
							0 \arrow[r] &  U \arrow[r] \arrow[d]          &\mathfrak{X} (E_{p^{\infty}}/L_\infty)  \arrow[r] \arrow[d]    & \mathfrak{X}^{\pm/\pm} (E_{p^{\infty}}/L_\infty) \arrow[d] \arrow[r] & 0\\
							0 \arrow[r] &  U \ot K \arrow[r]         &\mathfrak{X} (E_{p^{\infty}}/L_\infty) \ot K  \arrow[r]     & \mathfrak{X}^{\pm/\pm} (E_{p^{\infty}}/L_\infty) \ot K \arrow[r] & 0.       	
						\end{tikzcd}\]
					}
					
					The snake lemma implies the following exact sequence  of $ \Lambda(G) $-modules:
					\begin{align}\label{eq:bounding_mu_invariants}
						\begin{tikzcd}
							0 \rightarrow 
							T_{\Lambda(G)}(\mathfrak{X}(\mathsf{E}_{p^{\infty}}/L_\infty)) 
							\rightarrow   \mathfrak{X}^{\pm/\pm}(\mathsf{E}_{p^{\infty}}/L_\infty)
							\rightarrow
							U \ot_{\Lambda(G)} K/\Lambda(G)
							\\
							\rightarrow  
							\mathfrak{X}(\mathsf{E}_{p^{\infty}}/L_\infty) \ot_{\Lambda(G)} K/\Lambda(G)
							\rightarrow  
							0.
						\end{tikzcd}
					\end{align}
					The exact sequence \eqref{eq:bounding_mu_invariants} implies that the $ \Lambda(G) $-module $ 			T_{\Lambda(G)}(\mathfrak{X}(\mathsf{E}_{p^{\infty}}/L_\infty))  $ embeds into $ \mathfrak{X}^{\pm/\pm}(\mathsf{E}_{p^{\infty}}/L_\infty) $.
					This means that 
					\[
					\mu_G\left(T_{\Lambda(G)}(\mathfrak{X}(\mathsf{E}_{p^{\infty}}/L_\infty))\right) \leq	\mu_G^{\pm/\pm}(E_{p^{\infty}}/L_\infty).
					\] 
				\end{proof}
				\begin{rem}
					A similar argument shows this result for the cyclotomic $ \bZ_p $-extensions. 
					Also, note that Proposition \ref{prop:signed_mu_implies_mu-vanishing} does not assume  \hyperref[Conjecture A]{Conjecture A}.
				\end{rem}
\bibliographystyle{line}
\bibliography{JAMS-paper}
			\end{document}